\documentclass[sn-mathphys]{sn-jnl}

\usepackage{amsmath,amsfonts,amsopn,amssymb}
\usepackage{graphicx,subfig}
\usepackage{multirow,relsize}

\newcommand\gap{\hspace{0.1cm}}

\newcommand\tx{\tilde{x}}
\newcommand\bepsilon{\bar{\epsilon}}
\newcommand\bdelta{\bar{\delta}}

\let\div\relax
\DeclareMathOperator{\div}{div}
\DeclareMathOperator*{\argmin}{\arg\min}
\DeclareMathOperator{\dom}{dom}

\jyear{2021}

\numberwithin{equation}{section}

\theoremstyle{thmstyleone}
\newtheorem{theorem}{Theorem}[section]
\newtheorem{proposition}[theorem]{Proposition}
\newtheorem{lemma}[theorem]{Lemma}
\newtheorem{claim}[theorem]{Claim}

\theoremstyle{thmstyletwo}

\newtheorem{remark}[theorem]{Remark}

\theoremstyle{thmstylethree}
\newtheorem{definition}[theorem]{Definition}
\newtheorem{assumption}[theorem]{Assumption}

\begin{document}

\title[Fast Gradient Methods for Uniformly Convex]{Fast Gradient Methods for Uniformly Convex and Weakly Smooth Problems}

\author*{\fnm{Jongho} \sur{Park}}\email{jongho.park@kaist.ac.kr}

\affil{\orgdiv{Natural Science Research Institute}, \orgname{KAIST}, \orgaddress{\city{Daejeon}, \postcode{34141},\country{Korea}}}


\abstract{In this paper, acceleration of gradient methods for convex optimization problems with weak levels of convexity and smoothness is considered.
Starting from the universal fast gradient method which was designed to be an optimal method for weakly smooth problems whose gradients are H\"{o}lder continuous, its momentum is modified appropriately so that it can also accommodate uniformly convex and weakly smooth problems.
Different from the existing works, fast gradient methods proposed in this paper do not use the restarting technique but use momentums that are suitably designed to reflect both the uniform convexity and weak smoothness information of the target energy function.
Both theoretical and numerical results that support the superiority of the proposed methods are presented.}

\keywords{gradient methods, uniform convexity, weak smoothness, momentum acceleration, convex optimization}

\pacs[MSC Classification]{90C25, 68Q25, 65K05, 65B99}

\maketitle

\section{Introduction}
\label{Sec:Introduction}
We consider first-order methods for convex optimization problems of the form
\begin{equation}
\label{model0}
\min_{x \in X} F(x),
\end{equation}
where $X$ is an Euclidean space equipped with a norm $\| \cdot \| = \left< \cdot , \cdot \right>^{\frac{1}{2}}$, and $F \colon X \rightarrow \overline{\mathbb{R}}$ is a convex function on $X$.
We further assume that $F$ is coercive, so that~\eqref{model0} admits a solution~\cite[Proposition~11.14]{BC:2011}.
It is well-known that the worst-case convergence rate of a first-order method for~\eqref{model0} highly depends on the levels of smoothness and convexity of the objective function $F$; see, e.g.,~\cite{NN:1985,Park:2020,RD:2020}.
For the sake of clarity, we first present the definitions of uniform convexity and weak smoothness of a function.

\begin{definition}
\label{Def:convex}
Let $K$ be a compact convex subset of $X$.
A differentiable function $h \colon X \rightarrow \mathbb{R}$ is $(p, \mu)$-uniformly convex on $K$ if there exists two constants $p \geq 1$ and $\mu > 0$ such that
\begin{equation*}
h(x) \geq h(y) + \left< \nabla h(y), x-y \right> + \frac{\mu}{p} \| x-y \|^p, \quad x,y \in K.
\end{equation*}
\end{definition}

\begin{definition}
\label{Def:smooth}
Let $K$ be a compact convex subset of $X$.
A differentiable function $h \colon X \rightarrow \mathbb{R}$ is $(q, L)$-weakly smooth on $K$ if there exists two constants $q \geq 1$ and $L > 0$ such that
\begin{equation*}
h(x) \leq h(y) + \left< \nabla h(y), x-y \right> + \frac{L}{q} \| x-y\|^q, \quad x,y \in K.
\end{equation*}
\end{definition}

Definition~\ref{Def:convex} reduces to the notion of strong convexity if $p = 2$.
A remarkable property related to Definition~\ref{Def:convex} is the H\"{o}lderian error bound~\cite{RD:2020}; if $K$ contains the solution set $X^*$ of~\eqref{model0} and $h$ is $(p, \mu)$-uniformly convex on $K$, then it follows that
\begin{equation}
\label{sharp}
h(x) \geq h(x^*) + \frac{\mu}{p} \|x - x^* \|^p, \quad x \in K,\gap x^* \in X^*.
\end{equation} 
Inequalities of the type~\eqref{sharp} appear in a very broad class of functions; see, e.g.,~\cite{BDL:2007,XY:2013}.
On the other hand, Definition~\ref{Def:smooth} is a necessary condition for the $(q-1)$-H\"{o}lder continuity of the gradient of a function~\cite{Nesterov:2015}, i.e.,
\begin{equation}
\label{Holder}
\| \nabla h(x) - \nabla h(y) \| \leq L \| x - y \|^{q-1}, \quad x,y \in K.
\end{equation}  
If $q=2$,~\eqref{Holder} reduces to the Lipschitz continuity of the gradient, which is a fundamental assumption in design of first-order methods for convex optimization; see, e.g.,~\cite{CP:2016,Nesterov:2018}.
We also note that the notions introduced in Definitions~\ref{Def:convex} and~\ref{Def:smooth} are typical in the theory of Banach spaces~\cite{AP:1995,XR:1991}.

An important example satisfying Definitions~\ref{Def:convex} and~\ref{Def:smooth} from structural mechanics is the $s$-Laplacian problem~\cite{Ciarlet:2002,Park:2020}.
For $s \geq 1$, a solution of the $s$-Laplacian equation
\begin{align*}
- \div \left( \lvert \nabla u \rvert^{s-2} \nabla u \right) &= b \quad \textrm{ in } \Omega, \\
u&=0 \quad \textrm{ on } \partial \Omega
\end{align*}
is characterized by a unique minimizer of the convex optimization problem
\begin{equation}
\label{sLap}
\min_{u \in W_0^{1,s} (\Omega)} \left\{ F(u) := \frac{1}{s} \int_{\Omega} \lvert \nabla u \rvert^s \,dx - \int_{\Omega} bu \,dx \right\},
\end{equation}
where $\Omega$ is a polygonal domain in $\mathbb{R}^d$, $b \in L^{\frac{s}{s-1}}(\Omega)$, and the solution space $W_0^{1,s}(\Omega)$ is the collection of all functions in $L^s (\Omega)$ with vanishing trace on $\partial \Omega$ such that $\nabla u \in \left(L^s (\Omega) \right)^d$.
It is well-known that the energy function $F$ of~\eqref{sLap} satisfies Definitions~\ref{Def:convex} and~\ref{Def:smooth} with $p= \max \left\{2, s \right\}$ and $q = \min \left\{2, s \right\}$, respectively~\cite{Ciarlet:2002,Park:2020}.
Due to the importance of the $s$-Laplacian in mathematical modeling, there has been extensive research on fast solvers for $s$-Laplacian problems~\cite{BI:2000,FSWW:2017,HLL:2007,Park:2020,ZF:2013}.

This paper is devoted to design of acceleration schemes for gradient methods for the general convex optimization problem~\eqref{model0} satisfying Definitions~\ref{Def:convex} and~\ref{Def:smooth}.
We assume that the objective function $F$ of~\eqref{model0} can be decomposed into the sum of two convex functions as follows:
\begin{equation}
\label{model}
\min_{x \in X} \left\{ F(x) := f(x) + g(x) \right\},
\end{equation}
where $f \colon X \rightarrow \mathbb{R}$ is a differentiable convex function, and $g \colon X \rightarrow \overline{\mathbb{R}}$ is a convex lower semicontinuous function which is possibly nonsmooth.
We further assume that~\eqref{model} admits a solution $x^* \in \dom F$.

Starting from the celebrated work of Nesterov~\cite{Nesterov:1983}, designing acceleration schemes for first-order methods for convex optimization problems has been one of the most promising topics in mathematical optimization.
In~\cite{Nesterov:2005,Nesterov:1983}, Nesterov introduced the notion of momentum acceleration in order to obtain fast gradient methods for smooth convex optimization, i.e., the case when $f$ satisfies Definition~\ref{Def:smooth} with $q=2$ and $g =0$ in~\eqref{model}.
Momentum acceleration was successfully applied to the general nonsmooth case~($g \neq 0$) in~\cite{BT:2009,Nesterov:2013}.
Based on the notion of inexact oracle~\cite{DGN:2013,DGN:2014}, the universal fast gradient method was proposed in~\cite{Nesterov:2015} in order to deal with weakly smooth problems; the function $f$ satisfies Definition~\ref{Def:smooth} with $q \leq 2$.
In~\cite{IN:2014}, first-order methods for uniformly convex objectives satisfying Definition~\ref{Def:convex} with $p \geq 1$ were considered.
Meanwhile, there has been proposed a methodology called the performance estimation problem approach~\cite{DT:2014} that optimizes momentum acceleration.
Using the performance estimation problem approach, optimized gradient methods that achieve smaller convergence bounds than Nesterov's methods were proposed~\cite{KF:2016,KF:2018}.

If the energy function $F$ of~\eqref{model} is uniformly convex, i.e., it satisfies Definition~\ref{Def:convex} for some $p$, then improved convergence bounds can be obtained for first-order methods in general~\cite{Park:2020,RD:2020}.
In this case, in order to accelerate gradient methods for~\eqref{model} properly, the momentum must be designed elaborately so that it reflects the uniform convexity information of the energy function. 
For the strongly convex case, i.e., when $p=2$, there have been proposed several remarkable fast gradient methods that use appropriate momentums~\cite{CC:2019,CP:2016,Nesterov:2013}.
An alternative approach to momentum acceleration to deal with uniformly convex objectives is the restarting technique; if fast gradient methods designed for nonuniformly convex problems are restarted at some iterations, then they can achieve optimal convergence bounds for uniformly convex problems~\cite{NN:1985}.
Restarting techniques for strongly convex problems were considered in~\cite{Nesterov:2013}.
In~\cite{OC:2015}, heuristic restarting techniques that do not require a prior spectral information of problems were proposed.
A restarting technique called scheduled restarts was proposed in~\cite{RD:2020} and it showed optimal convergence properties for convex optimization problems with general $p$ and $q$.
Very recently, another novel restarting scheme that is adaptive to the levels of convexity and smoothness was considered in~\cite{RG:2021}.
Meanwhile, there are several recent results on fast gradient methods for general $p$ and $q$ that utilize momentums~\cite{NGGD:2020,Stonyakin:2021}.
In~\cite{NGGD:2020}, a universal fast gradient method for strongly convex problems without restarts was proposed.
In~\cite{Stonyakin:2021}, a universal fast gradient method based on a framework of inexact model was considered and it can be applied to any levels of convexity and smoothness. 
However, to the best of our knowledge, there have been no existing works on accelerated gradient methods without restarts enjoying optimal convergence rates with respect to $p$ and $q$~\cite{DGN:2013,RD:2020}.

In this paper, we focus on how to design momentums for gradient methods that suitably reflect the uniform convexity and the weak smoothness information of the energy function.
Our goal is motivated by~\cite[Figure~3]{OC:2015}; it presents that, for the strongly convex case, a fast gradient method that uses an optimal momentum converges faster than methods using the restarting technique.
Hoping that this phenomenon may be generalized to the general uniformly convex case, we construct novel momentums that are suitable for general $p$ and $q$.
The starting point is the universal fast gradient method~\cite{Nesterov:2015} that shows an optimal convergence rate for weakly smooth problems with general $q$.
Proceeding similarly to~\cite{Nesterov:2013}, we generalize the universal fast gradient method so that it becomes also suitable for strongly convex problems.
Then, by approximating uniform convexity of the energy function to strong convexity with some tolerance~(cf.~\cite[Theorem~3]{DGN:2013}), we obtain fast gradient methods that perform well for convex optimization problems with general $p$ and $q$.
Numerical results show that proposed methods has faster convergence rates than state-of-the-art methods such as the universal scheduled restarts~\cite{RD:2020}.
We also provide theoretical results that proposed methods are optimal up to a logarithmic factor~\cite{DGN:2013,RD:2020}.

The rest of this paper is organized as follows.
We briefly review several existing approaches~\cite{Nesterov:2015,RD:2020} to convex optimization under the uniform convexity and weak smoothness assumptions in Section~\ref{Sec:UGM}.
Fast gradient methods for weakly smooth and strongly convex problems are introduced in Section~\ref{Sec:Strong}, and then they are extended to more general uniformly convex problems in Section~\ref{Sec:Uniform}.
Numerical results for fast gradient methods introduced in this paper are presented in Section~\ref{Sec:Numerical}.
We conclude the paper with remarks in Section~\ref{Sec:Conclusion}.

\section{Universal fast gradient method}
\label{Sec:UGM}
In this section, we review the universal fast gradient method proposed by Nesterov~\cite{Nesterov:2015}.
Then we summarize the scheduled restarting technique~\cite{RD:2020} for the universal fast gradient method.
Throughout this paper, we take $x_0 \in \dom F$ as an initial guess for algorithms for solving the model problem~\eqref{model}.
A subset $K_0$ of $X$ is defined in terms of $x_0$ as follows:
\begin{equation}
\label{K0}
K_0 = \left\{ x \in X : F(x) \leq F(x_0) \right\}.
\end{equation}
We clearly have $K_0 \subset \dom F$.
Moreover, $K_0$ is compact since $F$ is coercive and lower semicontinuous.
The universal fast gradient method is a first-order method to solve~\eqref{model} under the following assumption.

\begin{assumption}
\label{Ass:smooth}
In~\eqref{model}, the function $f$ is $(q,L)$-weakly smooth on $K_0$ for some $1 \leq q \leq 2$ and $L > 0$, where $K_0$ was defined in~\eqref{K0}.
\end{assumption}

One may refer to~\cite{Nesterov:2015} for various examples satisfying Assumption~\ref{Ass:smooth}.
In what follows, we write
\begin{equation*}
\ell_F (x;y) = f(y) + \left< \nabla f(y), x - y \right> + g(x),
\end{equation*}
for $x,y \in \dom F$, i.e., $\ell_F(\cdot; y)$ is a partial linearization of $F$ at $y$.
The universal fast gradient method equipped with the monotonicity-enforcing technique proposed in~\cite{RD:2020} is presented in Algorithm~\ref{Alg:original}.
Enforcing monotonicity of Algorithm~\ref{Alg:original} is for the sake of ensuring that the energy sequence generated by the algorithm is contained in $K_0$.

\begin{algorithm}[]
\caption{Universal fast gradient method for~\eqref{model} under Assumption~\ref{Ass:smooth}}
\begin{algorithmic}[]
\State Choose $x_0 \in X$, $L_0 > 0$, and $\epsilon > 0$.
\State Let $A_0 = 0$ and $\phi_0 (x) = \frac{1}{2} \| x - x_0 \|^2$.

\For{$n=0,1,2,\dots$}
\State \begin{equation*}
\hat{L}_{n+1} \leftarrow L_n / 2, \quad
v_n = \argmin_{x \in X} \phi_n (x)
\end{equation*}

\Repeat
\State Find $a_{n+1} \geq 0$ such that $\displaystyle \frac{a_{n+1}^2}{A_n + a_{n+1}} = \frac{1}{\hat{L}_{n+1}}$.
\State \begin{equation} 
\label{common_subroutine}
\begin{split}
\theta_n &= \frac{a_{n+1}}{A_n + a_{n+1}} \\
y_n &= (1 - \theta_n) x_n + \theta_n v_n \\
z_n &= \argmin_{x \in X} \left\{\ell_F ( x; y_n) + \frac{\theta_n \hat{L}_{n+1}}{2} \| x - v_n \|^2 \right\} \\
\tx_{n+1} &= (1-\theta_n) x_{n} + \theta_n z_n
\end{split}\end{equation}

\If {$\displaystyle F(\tx_{n+1}) > \ell_F(\tx_{n+1}; y_n) + \frac{\hat{L}_{n+1}}{2} \| \tx_{n+1} - y_n \|^2 + \frac{\theta_n \epsilon}{2}$}
\State \begin{equation*}
\hat{L}_{n+1} \leftarrow 2\hat{L}_{n+1}
\end{equation*}
\EndIf
\Until {$\displaystyle F(\tx_{n+1}) \leq \ell_F(\tx_{n+1}; y_n) + \frac{\hat{L}_{n+1}}{2} \| \tx_{n+1} - y_n \|^2 + \frac{\theta_n\epsilon}{2}$}
\State Pick $x_{n+1}$ such that $\displaystyle F(x_{n+1}) \leq \min \left\{ F(\tx_{n+1}), F(x_n) \right\}$.
\begin{equation*}
L_{n+1} = \hat{L}_{n+1}, \quad 
A_{n+1} = A_n + a_{n+1}, \quad
\phi_{n+1}(x) = \phi_{n}(x) + a_{n+1} \ell_F (x; y_n)
\end{equation*} 
\EndFor
\end{algorithmic}
\label{Alg:original}
\end{algorithm}

Inputs $L_0$ and $\epsilon$ of Algorithm~\ref{Alg:original} play roles of estimates for the smoothness parameter $L$ and target accuracy, respectively.
In addition, $\epsilon$ plays a role of inexactness of oracles~\cite{DGN:2013,DGN:2014}.
A remarkable property of Algorithm~\ref{Alg:original} is that it does not require a priori information on the exact values of $q$ and $L$; this is why the method is said to be ``universal''~\cite{Nesterov:2015}.
The following convergence theorem is available~\cite[Theorem~3]{Nesterov:2015}.

\begin{proposition}
\label{Prop:original}
In Algorithm~\ref{Alg:original}, we have
\begin{equation*}
F(x_n ) - F(x^* ) \leq  \frac{\| x_0 - x^* \|^2}{2A_n} + \frac{\epsilon}{2},
\quad n \geq 1,
\end{equation*}
where
\begin{equation*}
A_n \geq \frac{\epsilon^{\frac{2-q}{q}} n^{\frac{3q-2}{q}}}{2^{\frac{4q-2}{q}} L^{\frac{2}{q}}} ,
\quad n \geq 1.
\end{equation*}
In addition, we need at most
\begin{equation*}
\frac{2^{\frac{4q-2}{3q-2}} L^{\frac{2}{3q-2}} \| x_0 - x^* \|^{\frac{2q}{3q-2}}}{\epsilon^{\frac{2}{3q-2}}}
\end{equation*}
iterations to obtain an $\epsilon$-solution of~\eqref{model}.
\end{proposition}

Proposition~\ref{Prop:original} says that $O (\epsilon^{-\frac{2}{3q-2}})$ iterations of Algorithm~\ref{Alg:original} are enough to obtain an $\epsilon$-solution of~\eqref{model} for any given target accuracy $\epsilon$, even though the energy error $F(x_n) - F(x^*)$ may stagnate at some level and not converge to zero.
Such a complexity estimate is optimal for first-order methods under Assumption~\ref{Ass:smooth}~\cite{NN:1985}.
On the other hand, it is well-known that the lower bound for the number of iterations to obtain an $\epsilon$-solution can be improved if the objective function satisfies the following additional assumption on uniform convexity.

\begin{assumption}
\label{Ass:convex}
In~\eqref{model}, the function $f$ is $(p, \mu)$-uniformly convex on $K_0$ for some $p \geq 2$ and $\mu \geq 0$, where $K_0$ was defined in~\eqref{K0}.
\end{assumption}

Under Assumptions~\ref{Ass:smooth} and~\ref{Ass:convex}, one can define the generalized condition number $\kappa$ of $f$ on $K_0$ as follows~\cite{RD:2020}:
\begin{equation}
\label{kappa}
\kappa = \frac{L^{\frac{2}{q}}}{\mu^\frac{2}{p}}.
\end{equation}
It was observed in, e.g.,~\cite{Park:2020,RD:2020}, that convergence behaviors of first-order methods for~\eqref{model} depends on $\kappa$; the greater $\kappa$ is, the slower the convergence rate is in general.

\begin{algorithm}[]
\caption{Universal scheduled restarts for~\eqref{model} under Assumptions~\ref{Ass:smooth} and~\ref{Ass:convex}}
\begin{algorithmic}[]
\State Choose $x_0 \in X$, $L_0 > 0$, $\epsilon_0 > 0$, $\gamma \geq 0$, and $\{ t_k > 0 \}_{k\geq 1}$.
\State Let $A_0 = 0$ and $\phi_0 (x) = \frac{1}{2} \| x - x_0 \|^2$.

\For{$n=0,1,2,\dots$}
\State \begin{equation*}
\hat{L}_{n+1} \leftarrow L_n / 2, \quad
v_n = \argmin_{x \in X} \phi_n (x)
\end{equation*}

\Repeat
\State Find $a_{n+1} \geq 0$ such that $\displaystyle \frac{a_{n+1}^2}{A_n + a_{n+1}} = \frac{1}{\hat{L}_{n+1}}$.
\State Compute $\theta_n$, $y_n$, $z_n$, and $\tx_{n+1}$ by~\eqref{common_subroutine}.

\If {$\displaystyle F(\tx_{n+1}) > \ell_F(\tx_{n+1}; y_n) + \frac{\hat{L}_{n+1}}{2} \| \tx_{n+1} - y_n \|^2 + \frac{\theta_n \epsilon_n}{2}$}
\State \begin{equation*}
\hat{L}_{n+1} \leftarrow 2\hat{L}_{n+1}
\end{equation*}
\EndIf
\Until {$\displaystyle F(\tx_{n+1}) \leq \ell_F(\tx_{n+1}; y_n) + \frac{\hat{L}_{n+1}}{2} \| \tx_{n+1} - y_n \|^2 + \frac{\theta_n\epsilon_n}{2}$}

\State Pick $x_{n+1}$ such that $\displaystyle F(x_{n+1}) \leq \min \left\{ F(\tx_{n+1}), F(x_n) \right\}$.
\begin{equation*}
L_{n+1} = \hat{L}_{n+1}
\end{equation*}

\If {$\displaystyle n+1 = \sum_{k=1}^r \lceil t_k \rceil$ for some $r \geq 1$}
\State \begin{equation*}
\epsilon_{n+1} = e^{-\gamma} \epsilon_n, \quad
A_{n+1} = 0, \quad
\phi_{n+1} (x) = \frac{1}{2} \| x - x_{n+1} \|^2
\end{equation*}
\Else
\State \begin{equation*}
\epsilon_{n+1} = \epsilon_{n}, \quad
A_{n+1} = A_n + a_{n+1}, \quad
\phi_{n+1}(x) = \phi_{n}(x) + a_{n+1} \ell_F (x; y_n)
\end{equation*} 
\EndIf
\EndFor
\end{algorithmic}
\label{Alg:scheduled}
\end{algorithm}

A common approach to achieve the optimal energy convergence rate for uniformly convex problems is the restarting technique~\cite{NN:1985,Nesterov:2013,OC:2015,RD:2020}.
In~\cite{RD:2020}, the scheduled restarting technique for the universal fast gradient method was proposed; for a prescribed restarting schedule $\{t_k > 0 \}_{k \geq 1}$, Algorithm~\ref{Alg:original} is restarted at each $t_k$ iterations.
The universal fast gradient method equipped with scheduled restarts is summarized in Algorithm~\ref{Alg:scheduled}.
As stated in Proposition~\ref{Prop:scheduled}, Algorithm~\ref{Alg:scheduled} achieves the optimal energy convergence if a restarting schedule $\{ t_k > 0 \}_{k \geq 1}$ and a sequence of tolerances $\{ \epsilon_n \}_{n \geq 0}$ are suitably chosen~\cite[Proposition~3.1 and Lemma~B.1]{RD:2020}.

\begin{proposition}
\label{Prop:scheduled}
In Algorithm~\ref{Alg:scheduled}, we choose
\begin{equation*}
\epsilon_0 \geq e^{-\gamma} \left( F(x_0) - F(x^*) \right), \quad
\gamma = \frac{3q-2}{2}, \quad
t_k = C e^{\left( 1 - \frac{q}{p} \right)k},
\end{equation*}
where
$
C = e^{\frac{q}{p}} (8e^{\frac{2}{e}} \kappa )^{\frac{q}{3q-2}} (e^{\gamma}\epsilon_0)^{- \frac{2(p-q)}{p(3q-2)}}
$
and $\kappa$ was defined in~\eqref{kappa}.
If $n = \sum_{k=1}^r \lceil t_k \rceil$ for some $r \geq 1$, we have the following:
\begin{enumerate}
\item In the case $p=q$, we have
\begin{equation*}
F(x_n ) - F(x^* ) = O \left( \epsilon_0 \exp \left( - \kappa^{-\frac{q}{3q-2}} n \right) \right) .
\end{equation*}
\item In the case $p > q$, we have
\begin{equation*}
F(x_n) - F(x^*) = O \left( \epsilon_0 \kappa^{-\frac{pq}{2(p-q)}} n^{-\frac{p(3q-2)}{2(p-q)}} \right).
\end{equation*}
\end{enumerate}
\end{proposition}

\begin{figure}[]
\centering
\subfloat[][Varying $\epsilon_0$, $C=2$]{ \includegraphics[width=0.45\linewidth]{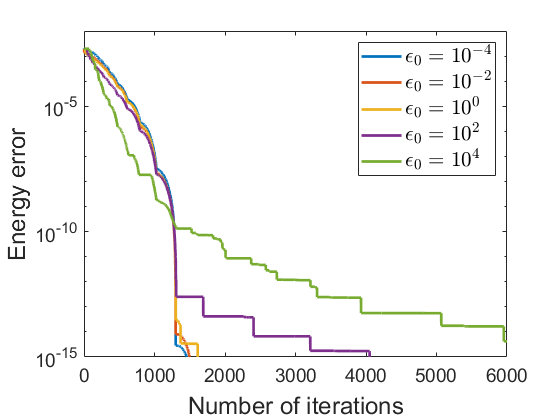} }
\subfloat[][Varying $C$, $\epsilon_0 = 10^{-2}$]{ \includegraphics[width=0.45\linewidth]{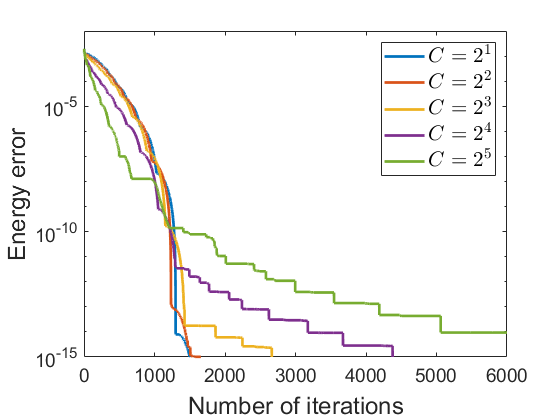} }
\caption{Decay of the energy error $F(x_n) - F(x^*)$  of Algorithm~\ref{Alg:scheduled}~($x_0 = 0$, $L_0 = 1$, $\gamma = \frac{3q-2}{2}$) for~\eqref{sLap_FEM}~($s=1.5$, $b=1$, $h=2^{-5}$) with respect to \textbf{(a)} various $\epsilon_0$ and \textbf{(b)} various $C$.
In this case, we have $p=2$ and $q=1.5$.}
\label{Fig:scheduled}
\end{figure}

In spite of the optimal convergence property of Algorithm~\ref{Alg:scheduled} presented in Proposition~\ref{Prop:scheduled}, it has a disadvantage that its convergence rate is highly sensitive on a choice of parameters $\epsilon_0$ and $C$, where $t_k = C e^{\left(1- \frac{q}{p} \right)k}$ for all $k \geq 1$.
Figure~\ref{Fig:scheduled} shows the convergence behavior of Algorithm~\ref{Alg:scheduled} for various values of $\epsilon_0$ and $C$.
One can readily observe from Figure~\ref{Fig:scheduled} that the convergence rate of Algorithm~\ref{Alg:scheduled} deteriorates critically if either $\epsilon_0$ or $C$ is chosen not properly; detailed settings on the numerical experiments corresponding to Figure~\ref{Fig:scheduled} will be presented in Section~\ref{Sec:Numerical}.
In order to compute the theoretically guaranteed values of $\epsilon_0$ and $C$ given in Proposition~\ref{Prop:scheduled}, one requires prior information on $F(x^*)$, $L$, and $\mu$, which is a quite restrictive situation.
For practical uses of Algorithm~\ref{Alg:scheduled}, grid searches on $\epsilon_0$ and $C$ must be accompanied.
We note that, very recently, a nearly-optimal restarting scheme that does not require a priori knowledge on the parameters was proposed in~\cite{RG:2021}, based on parallel executions of multiple copies of first-order methods.

\section{Acceleration for strongly convex problems}
\label{Sec:Strong}
As an alternative approach to the restarting technique, one may consider designing suitable momentum acceleration schemes that reflect the uniform convexity of the objective function~\cite{CC:2019,CP:2016,Nesterov:2013,Nesterov:2018}.
In this section, we propose fast gradient methods with novel momentum techniques for the composite optimization~\eqref{model} under Assumption~\ref{Ass:smooth} and the following additional strong convexity assumption.

\begin{assumption}
\label{Ass:strong}
In~\eqref{model}, the function $f$ is $(2, \mu)$-uniformly convex on $K_0$ for some $\mu \geq 0$, where $K_0$ was defined in~\eqref{K0}.
\end{assumption}

Throughout this section, we assume that a priori information on the strong convexity parameter $\mu$ is available, as it is a common assumption on designing momentum acceleration schemes for strongly convex problems~\cite{CC:2019,CP:2016,Nesterov:2013,Nesterov:2018}.
We mention that there have been proposed several accelerated first-order methods adaptive to the unknown strong convexity parameter; see, e.g.,~\cite{FQ:2019,Nesterov:2013}.
Motivated by~\cite{Nesterov:2013}, we present a generalized version of Algorithm~\ref{Alg:original} for strongly convex objectives in  Algorithm~\ref{Alg:strong0}.
It is clear that Algorithm~\ref{Alg:strong0} reduces to Algorithm~\ref{Alg:original} if we set $\mu = 0$.

\begin{algorithm}[]
\caption{Fast gradient method for~\eqref{model} under Assumptions~\ref{Ass:smooth} and~\ref{Ass:strong}}
\begin{algorithmic}[]
\State Choose $x_0 \in X$, $L_0 > 0$, and $\epsilon > 0$.
\State Let $A_0 = 0$ and $\phi_0 (x) = \frac{1}{2} \| x - x_0 \|^2$.

\For{$n=0,1,2,\dots$}
\State \begin{equation*}
\hat{L}_{n+1} \leftarrow L_n / 2, \quad
v_n = \argmin_{x \in X} \phi_n (x)
\end{equation*}

\Repeat
\State Find $a_{n+1} \geq 0$ such that
\begin{equation}
\label{A_n_strong0}
\frac{a_{n+1}^2}{A_n + a_{n+1}} = \frac{1 + \mu A_n}{\hat{L}_{n+1}}.
\end{equation}
\State Compute $\theta_n$, $y_n$, $z_n$, and $\tx_{n+1}$ by~\eqref{common_subroutine}.

\If {$\displaystyle F(\tx_{n+1}) > \ell_F(\tx_{n+1}; y_n) + \frac{\hat{L}_{n+1}}{2} \| \tx_{n+1} - y_n \|^2 + \frac{\theta_n \epsilon}{2}$}
\State \begin{equation*}
\hat{L}_{n+1} \leftarrow 2\hat{L}_{n+1}
\end{equation*}
\EndIf
\Until {$\displaystyle F(\tx_{n+1}) \leq \ell_F(\tx_{n+1}; y_n) + \frac{\hat{L}_{n+1}}{2} \| \tx_{n+1} - y_n \|^2 + \frac{\theta_n\epsilon}{2}$}
\State Pick $x_{n+1}$ such that $\displaystyle F(x_{n+1}) \leq \min \left\{ F(\tx_{n+1}), F(x_n) \right\}$.
\begin{gather*} 
L_{n+1} = \hat{L}_{n+1}, \quad 
A_{n+1} = A_n + a_{n+1}, \\
\phi_{n+1}(x) = \phi_{n}(x) + a_{n+1} \left( \ell_F (x; y_n) + \frac{\mu}{2} \| x - y_n \|^2 \right)
\end{gather*} 
\EndFor
\end{algorithmic}
\label{Alg:strong0}
\end{algorithm}

One may regard Algorithm~\ref{Alg:strong0} as a combination of Algorithm~\ref{Alg:original} and the momentum strategy proposed in~\cite{Nesterov:2013} for strongly convex problems.
In Algorithm~\ref{Alg:strong0}, there is an Armijo-type backtracking process~\cite{Armijo:1966} to find a suitable descent step size in terms of $L_n$.
The following lemma ensures that the backtracking process ends in finite steps, i.e., $L_n$ does not blow up to infinity. 

\begin{lemma}
\label{Lem:backtracking_strong0}
The backtracking process in Algorithm~\ref{Alg:strong0} terminates in finite steps.
Moreover, we have
\begin{equation*}
L_{n+1} \leq \frac{2L^{\frac{2}{q}}}{(\theta_n \epsilon)^{\frac{2-q}{q}}}, \quad n \geq 0.
\end{equation*}
\end{lemma}
\begin{proof}
See~\cite[Proposition~A.2]{RD:2020}.
\end{proof}

Similar to Proposition~\ref{Prop:original}, one can derive a complexity estimate for the number of iterations of Algorithm~\ref{Alg:strong0} that is required to get an $\epsilon$-solution of~\eqref{model} as follows.

\begin{theorem}
\label{Thm:strong0}
In Algorithm~\ref{Alg:strong0}, we have
\begin{equation}
\label{error_strong0}
F(x_n ) - F(x^* ) \leq  \frac{\| x_0 - x^* \|^2}{2A_n} + \frac{\epsilon}{2},
\quad n \geq 1,
\end{equation}
where
\begin{equation}
\label{A_n_bound_strong0}
A_n \geq \frac{\epsilon^{\frac{2-q}{q}}}{2L^{\frac{2}{q}}} \left( 1 + \frac{\epsilon^{\frac{2-q}{3q-2}}}{2^{\frac{4q-2}{3q-2}} \kappa^{\frac{q}{3q-2}}} \right)^{\frac{3q-2}{q}(n-1)}, \quad n \geq 1.
\end{equation}
In addition, we need at most
\begin{equation*}
1 + \frac{q}{3q-2} \frac{\log \left( 2L^{\frac{2}{q}}  \| x_0 - x^* \|^2 \right) + \frac{2}{q} \log \frac{1}{\epsilon}}{\log \left( 1 + \frac{\epsilon^{\frac{2-q}{3q-2}}}{2^{\frac{4q-2}{3q-2}} \kappa^{\frac{q}{3q-2}}} \right)}
= O \left( \frac{\kappa^{\frac{q}{3q-2}}}{\epsilon^{\frac{2-q}{3q-2}}} \log \frac{1}{\epsilon} \right)
\end{equation*}
iterations to obtain an $\epsilon$-solution of~\eqref{model}.
\end{theorem}
\begin{proof}
As the inequality~\eqref{error_strong0} is a special case of~\eqref{Thm:strong}, we omit its proof.
To estimate a lower bound for $A_n$, we start from~\eqref{A_n_strong0}:
it follows by Lemma~\ref{Lem:backtracking_strong0} that
\begin{equation*}
( A_{n+1} - A_n )^2 = \frac{A_{n+1} ( 1 + \mu A_n )}{L_{n+1}}
\geq \frac{(\theta_n \epsilon)^{\frac{2-q}{q}}}{2\kappa} A_n A_{n+1} ,
\end{equation*}
or equivalently,
\begin{equation}
\label{recur_strong0}
(A_{n+1} - A_n)^{\frac{3q-2}{q}} \geq \frac{\epsilon^{\frac{2-q}{q}}}{2\kappa}  A_n A_{n+1}^{\frac{2q-2}{q}}.
\end{equation}
Meanwhile, substituting $n=0$ into~\eqref{A_n_strong0} yields
$
A_1 = \frac{1}{L_1 } \geq \frac{\epsilon^{\frac{2-q}{q}}}{2L^{\frac{2}{q}}}.
$
Combining~\eqref{recur_strong0} with Lemma~\ref{Lem:recur_linear}, we obtain~\eqref{A_n_bound_strong0}.

Thanks to~\eqref{error_strong0}, it suffices to solve the following equation with respect to $n$ in order to compute the number of iterations required to obtain an $\epsilon$-solution:
\begin{equation*}
\frac{\| x_0 - x^* \|^2}{2} = \frac{\epsilon}{2} \cdot \frac{\epsilon^{\frac{2-q}{q}}}{2L^{\frac{2}{q}}} \left( 1 + \frac{\epsilon^{\frac{2-q}{3q-2}}}{2^{\frac{4q-2}{3q-2}} \kappa^{\frac{q}{3q-2}}} \right)^{\frac{3q-2}{q}(n-1)}.
\end{equation*}
The solution of the above equation with respect to $n$ is given by
\begin{equation*} \begin{split}
n &= 1 + \frac{q}{3q-2} \frac{\log \left( 2L^{\frac{2}{q}}  \| x_0 - x^* \|^2 \right) + \frac{2}{q} \log \frac{1}{\epsilon}}{\log \left( 1 + \frac{\epsilon^{\frac{2-q}{3q-2}}}{2^{\frac{4q-2}{3q-2}} \kappa^{\frac{q}{3q-2}}} \right)} \\
&\leq 1 + \frac{q}{3q-2} \left( \log \left( 2L^{\frac{2}{q}}  \| x_0 - x^* \|^2 \right) + \frac{2}{q} \log \frac{1}{\epsilon} \right) \left( \frac{1}{2} + \frac{2^{\frac{4q-2}{3q-2}} \kappa^{\frac{q}{3q-2}}}{\epsilon^{\frac{2-q}{3q-2}}} \right),
\end{split} \end{equation*}
where the inequality is due to an elementary inequality~\cite{Love:1980}
\begin{equation}
\label{log_ineq}
\log \left( 1 + \frac{1}{t} \right) \geq \frac{2}{2t+1}, \quad t > 0.
\end{equation}
This completes the proof.
\end{proof}

\begin{figure}[]
\centering
\includegraphics[width=0.45\linewidth]{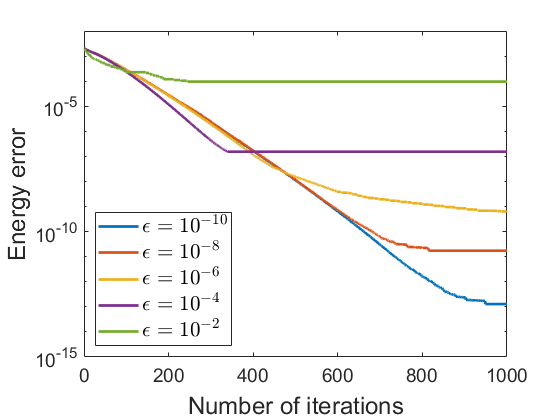} 
\caption{Decay of the energy error $F(x_n) - F(x^*)$ of Algorithm~\ref{Alg:strong0}~($x_0 = 0$, $L_0 = 1$) for~\eqref{sLap_FEM}~($s=1.5$, $b=1$, $h=2^{-5}$) with respect to various $\epsilon$.
In this case, we have $p=2$ and $q=1.5$.}
\label{Fig:strong0}
\end{figure}

As stated above, the iteration complexity of Algorithm~\ref{Alg:strong0} is $O ( \epsilon^{- \frac{2-q}{3q-2}} \log \epsilon^{-1})$, which is optimal under Assumptions~\ref{Ass:smooth} and~\ref{Ass:strong} up to a logarithmic factor~\cite{NN:1985}; we mention that a similar result was presented in~\cite{DGN:2013} before.
Another interesting point is the convergence behavior of the energy error $F(x_n) - F(x^*)$.
Equations~\ref{error_strong0} and~\ref{A_n_bound_strong0} implies that the energy error of Algorithm~\ref{Alg:strong0} decreases linearly until it stagnates at a particular value~\cite{Park:2021}. 
It is depicted in Figure~\ref{Fig:strong0} that as $\epsilon$ becomes smaller, the linear convergence rate becomes slower and the stagnation point becomes lower; see Section~\ref{Sec:Numerical} for details on the numerical experiments.
This motivates us to come up with an idea of choosing the tolerance $\epsilon$ ``adaptively.''
On the one hand, when the energy error is large, it seems good to set large $\epsilon$ so as to result fast decay of the energy error.
On the other hand, we have to choose small $\epsilon$ when the energy error is small to prevent stagnation of the energy.
In order to realize such an idea, we first present a generalized framework for Algorithm~\ref{Alg:strong0} that allows variable tolerance for each iteration; see Algorithm~\ref{Alg:strong}.

\begin{algorithm}[]
\caption{Fast gradient method for~\eqref{model} under Assumptions~\ref{Ass:smooth} and~\ref{Ass:strong} with variable tolerance}
\begin{algorithmic}[]
\State Choose $x_0 \in X$ and $L_0 > 0$.
\State Let $A_0 = 0$ and $\phi_0 (x) = \frac{1}{2} \| x - x_0 \|^2$.

\For{$n=0,1,2,\dots$}
\State \begin{equation*}
\hat{L}_{n+1} \leftarrow L_n / 2, \quad
v_n = \argmin_{x \in X} \phi_n (x)
\end{equation*}

\Repeat
\State Find $a_{n+1} \geq 0$ such that
\begin{equation}
\label{A_n_strong}
\frac{a_{n+1}^2}{A_n + a_{n+1}} = \frac{1 + \mu A_n}{\hat{L}_{n+1}}.
\end{equation}
\State Choose $\epsilon_n > 0$ and compute $\theta_n$, $y_n$, $z_n$, and $\tx_{n+1}$ by~\eqref{common_subroutine}.

\If {$\displaystyle F(\tx_{n+1}) > \ell_F(\tx_{n+1}; y_n) + \frac{\hat{L}_{n+1}}{2} \| \tx_{n+1} - y_n \|^2 + \frac{\theta_n \epsilon_n}{2}$}
\State \begin{equation*}
\hat{L}_{n+1} \leftarrow 2\hat{L}_{n+1}
\end{equation*}
\EndIf
\Until {$\displaystyle F(\tx_{n+1}) \leq \ell_F(\tx_{n+1}; y_n) + \frac{\hat{L}_{n+1}}{2} \| \tx_{n+1} - y_n \|^2 + \frac{\theta_n\epsilon_n}{2}$}
\State Pick $x_{n+1}$ such that $\displaystyle F(x_{n+1}) \leq \min \left\{ F(\tx_{n+1}), F(x_n) \right\}$.
\begin{gather*} 
L_{n+1} = \hat{L}_{n+1}, \quad 
A_{n+1} = A_n + a_{n+1}, \\
\phi_{n+1}(x) = \phi_{n}(x) + a_{n+1} \left( \ell_F (x; y_n) + \frac{\mu}{2} \| x - y_n \|^2 \right)
\end{gather*} 
\EndFor
\end{algorithmic}
\label{Alg:strong}
\end{algorithm}

The main feature of Algorithm~\ref{Alg:strong} is that the tolerance $\epsilon_n$ is newly determined at each iteration of the algorithm.
Different from the restart schedule $\{ t_k \}_{k \geq 1}$ in Algorithm~\ref{Alg:scheduled}, the sequence of tolerances $\{ \epsilon_n \}_{n \geq 0}$ in Algorithm~\ref{Alg:strong} need not to be determined prior to an execution of the algorithm.
Hence, $\epsilon_n$ may be determined based on some intermediate products of the algorithm up to the $n$th iteration such as $\{a_j\}_{1 \leq j \leq n+1}$, $\{ x_j \}_{1\leq j \leq  n}$, $\{ \epsilon_j \}_{1 \leq j \leq n-1}$, etc.
Note that if we choose $\epsilon_n = \epsilon > 0$ for all $n \geq 0$, then Algorithm~\ref{Alg:strong} reduces to Algorithm~\ref{Alg:strong0}.
The following theorem states an abstract energy error analysis for Algorithm~\ref{Alg:strong} with the general tolerance sequence.

\begin{theorem}
\label{Thm:strong}
In Algorithm~\ref{Alg:strong}, we have
\begin{equation*}
F(x_n ) - F(x^* ) \leq  \frac{\| x_0 - x^* \|^2}{2A_n} + \frac{\bepsilon_n}{2}, \quad n \geq 1,
\end{equation*}
where
$
\bepsilon_n = \frac{1}{A_n} \sum_{j=1}^n a_j \epsilon_{j-1}.
$
\end{theorem}
\begin{proof}
The proof can be done by the same argument as Theorem~\ref{Thm:uniform}.
\end{proof}

Next, we introduce several particular choices of $\{ \epsilon_n \}_{n \geq 0}$ that result satisfactorily fast decrease of the energy error.
As the first option, we choose
\begin{equation}
\label{e_strong_opt}
\epsilon_n = \frac{C}{a_{n+1} (A_n + a_{n+1})^{\frac{2-q}{3q-2}}}, \quad n \geq 0
\end{equation}
for some $C > 0$.
We will show that the energy error decay of Algorithm~\ref{Alg:strong}-\eqref{e_strong_opt} is optimal in the sense of Nemirovskii and Nesterov~\cite{NN:1985} up to a logarithmic factor.
In order to analyze the convergence behavior of Algorithm~\ref{Alg:strong}-\eqref{e_strong_opt}, one requires the following lemma which appeared in~\cite[Lemma~2]{Nesterov:2015} and~\cite[Lemma~A.4]{RD:2020}.

\begin{lemma}
\label{Lem:inexact_smooth}
Under Assumption~\ref{Ass:smooth}, for any $\epsilon > 0$, we have
\begin{equation*}
f(x) \leq f(y) + \left< \nabla f(y) , x-y \right> + \frac{L^{\frac{2}{q}}}{2\epsilon^{\frac{2-q}{q}}} \| x - y \|^2 + \frac{\epsilon}{2}, \quad x,y \in K_0.
\end{equation*}
\end{lemma}

\begin{remark}
\label{Rem:inexact_smooth}
We note that the following inequality which is a bit tighter result than Lemma~\ref{Lem:inexact_smooth} holds~\cite[Lemma~2]{Nesterov:2015}:
\begin{equation*}
f(x) \leq f(y) + \left< \nabla f(y) , x-y \right> + \frac{M_{\epsilon}}{2} \| x - y \|^2 + \frac{\epsilon}{2}, \quad x,y \in K_0,
\end{equation*}
where
\begin{equation*}
M_{\epsilon} = \left( \frac{2-q}{q} \frac{1}{\epsilon} \right)^{\frac{2-q}{q}} L^{\frac{2}{q}} \leq \frac{L^{\frac{2}{q}}}{\epsilon^{\frac{2-q}{q}}}
\end{equation*}
with the convention $0^0 = 1$.
However, we use Lemma~\ref{Lem:inexact_smooth} throughout the paper instead of the above one for the sake of simplicity.
\end{remark}
Using Lemma~\ref{Lem:inexact_smooth}, one can prove the finiteness of the backtracking process of Algorithm~\ref{Alg:strong}-\eqref{e_strong_opt}.
Since $\epsilon_n$ in~\eqref{e_strong_opt} is defined in terms of intermediate products $a_{n+1}$ and $A_n$, the argument of~\cite[Proposition~A.2]{RD:2020} cannot be applied in this case.

\begin{lemma}
\label{Lem:backtracking_strong_opt}
In Algorithm~\ref{Alg:strong}-\eqref{e_strong_opt}, the backtracking process terminates in finite steps.
Moreover, we have
\begin{equation}
\label{L_strong_opt}
L_{n+1} \leq \frac{2L^{\frac{2}{q}}}{(\theta_n \epsilon_n )^{\frac{2-q}{q}}}, \quad n \geq 0.
\end{equation}
\end{lemma}
\begin{proof}
Thanks to Lemma~\ref{Lem:inexact_smooth}, it suffices to show that $\hat{L}_{n+1} \geq \frac{L^{\frac{2}{q}}}{(\theta_n \epsilon_n )^{\frac{2-q}{q}}}$ for sufficiently large $\hat{L}_{n+1}$.
As the case $q=2$ is trivial, we may assume that $1 \leq q < 2$.
We first see that
\begin{equation*}
\theta_n \epsilon_n = \frac{a_{n+1}}{A_n + a_{n+1}} \frac{C}{a_{n+1} (A_n + a_{n+1})^{\frac{2-q}{3q-2}}} = \frac{C}{(A_n + a_{n+1})^{\frac{2q}{3q-2}}}.
\end{equation*}
Invoking~\eqref{A_n_strong}, we get
\begin{equation*} \begin{split}
\hat{L}_{n+1} \cdot \frac{(\theta_n \epsilon_n)^{\frac{2-q}{q}}}{L^{\frac{2}{q}}} &= \frac{(A_n + a_{n+1})(1+ \mu A_n )}{L^{\frac{2}{q}} a_{n+1}^2} \left( \frac{C}{(A_n + a_{n+1})^{\frac{2q}{3q-2}}} \right)^{\frac{2-q}{q}} \\
&\geq \frac{C^{\frac{2-q}{q}}}{L^{\frac{2}{q}}} \frac{\theta_n}{a_{n+1}^{\frac{q+2}{3q-2}}} \\
&= \frac{C^{\frac{2-q}{q}}}{L^{\frac{2}{q}}} \frac{1}{(A_n + a_{n+1}) a_{n+1}^{\frac{2(2-q)}{3q-2}}},
\end{split} \end{equation*}
where the inequality is due to $\frac{5q-6}{3q-2} \geq -1$.
Observing that $\frac{2(2-q)}{3q-2} > 0$ and that $a_{n+1}$ tends to $0$ as $\hat{L}_{n+1}$ tends to $\infty$~(see~\eqref{A_n_strong}), we deduce that $\hat{L}_{n+1} \cdot \frac{(\theta_n \epsilon_n)^{\frac{2-q}{q}}}{L^{\frac{2}{q}}}$ exceeds 1 for sufficiently large $\hat{L}_{n+1}$.
Meanwhile,~\eqref{L_strong_opt} is a direct consequence of Lemma~\ref{Lem:inexact_smooth}.
\end{proof}

Since the tolerance $\epsilon_n$ in~\eqref{e_strong_opt} decreases to 0 as $n$ increases, it is expected that the energy error tends to zero, not stagnating at some value like Algorithm~\ref{Alg:strong0}.
In the following theorem, we show that Algorithm~\ref{Alg:strong}-\eqref{e_strong_opt} enjoys $O(n^{- \frac{3q-2}{2-q}} \log n )$ convergence of the energy error when $1 \leq q < 2$, which is optimal up to a logarithmic factor under Assumptions~\ref{Ass:smooth} and~\ref{Ass:strong}~\cite{NN:1985}.

\begin{theorem}
\label{Thm:strong_opt}
In Algorithm~\ref{Alg:strong}-\eqref{e_strong_opt}, we have
\begin{equation}
\label{A_n_bound_strong_opt}
\displaystyle A_n \geq \begin{cases} \displaystyle \widetilde{C}n^{\frac{3q-2}{2-q}}, & 1 \leq q < 2, \\
\displaystyle \frac{1}{2L} \left( 1 + \frac{1}{2^{\frac{3}{2}} \kappa^{\frac{1}{2}}} \right)^{2(n-1)}, & q = 2, \end{cases}
\quad n \geq 1,
\end{equation}
where $\kappa$ was defined in~\eqref{kappa} and $\widetilde{C}$ is a positive constant depending on $q$, $L$, $\mu$, and $C$ only.
Consequently, it satisfies that
\begin{equation*}
F(x_n) - F(x^*) \leq \begin{cases} \displaystyle\frac{1}{2 \widetilde{C}n^{\frac{3q-2}{2-q}}} \left( \|x_0 - x^* \|^2 + \frac{C}{\widetilde{C}^{\frac{2-q}{3q-2}}} \left( 1 + \log n \right) \right), & 1 \leq q < 2, \\
\displaystyle L \left( \| x_0 - x^* \|^2 + Cn \right) \left( 1 + \frac{1}{2^{\frac{3}{2}} \kappa^{\frac{1}{2}}} \right)^{-2(n-1)}, & q = 2, \end{cases}
\gap n \geq 1.
\end{equation*}
\end{theorem}
\begin{proof}
Using~\eqref{A_n_strong}~\eqref{e_strong_opt}, and~\eqref{L_strong_opt}, we have the following:
\begin{equation}
\label{recur_strong_opt1}
(A_{n+1} - A_n)^2 \geq \frac{\mu A_n A_{n+1}}{L_{n+1}} 
\geq \frac{(\theta_n \epsilon_n)^{\frac{2-q}{q}}}{2\kappa}A_n {A_{n+1}} 
= \frac{C^{\frac{2-q}{q}}}{2\kappa} A_n A_{n+1}^{\frac{5q-6}{3q-2}}.
\end{equation}
If $q=2$,~\eqref{recur_strong_opt1} reduces to
\begin{equation*}
(A_{n+1} - A_n)^2 \geq \frac{1}{2\kappa} A_n A_{n+1},
\end{equation*}
so that invoking Lemma~\ref{Lem:recur_linear} yields
\begin{equation}
\label{A_n_bound_strong_opt1}
A_n \geq A_1 \left( 1 + \frac{1}{2^{\frac{3}{2}} \kappa^{\frac{1}{2}}} \right)^{2(n-1)}.
\end{equation}
On the other hand,~\eqref{L_strong_opt} implies that
$
A_1 = \frac{1}{L_1} \geq \frac{1}{2L}.
$
Invoking~\eqref{A_n_bound_strong_opt1} yields~\eqref{A_n_bound_strong_opt}.
Using~\eqref{A_n_bound_strong_opt} and Theorem~\ref{Thm:strong}, we conclude the following:
\begin{equation*} \begin{split}
F(x_n) - F(x^*) &\leq \frac{1}{2A_n} \left( \| x_0 - x^* \|^2 + \sum_{j=1}^n a_j \epsilon_{j-1} \right) \\
&\leq L \left( \| x_0 - x^* \|^2 + Cn \right) \left( 1 + \frac{1}{2^{\frac{3}{2}} \kappa^{\frac{1}{2}}} \right)^{-2(n-1)}.
\end{split} \end{equation*}

Next, we consider the case when $1 \leq q < 2$.
Since $A_{n+1} \geq A_n$, by~\eqref{recur_strong_opt1}, we get
\begin{equation}
\label{recur_strong_opt2}
A_{n+1} - A_n \geq \frac{C^{\frac{2-q}{2q}}}{2^{\frac{1}{2}}\kappa^{\frac{1}{2}}} A_n^{\frac{4q-4}{3q-2}}.
\end{equation}
Applying Lemma~\ref{Lem:recur_sublinear} to~\eqref{recur_strong_opt2} yields
\begin{equation}
\label{A_n_bound_strong_opt2}
A_n \geq \min \left\{ A_1, \left( \frac{C^{\frac{2-q}{2q}}}{2^{\frac{3}{2}}(2^{\frac{3q-2}{2-q}} - 1) \kappa^{\frac{1}{2}}}\right)^{\frac{3q-2}{2-q}} \right\} n^{\frac{3q-2}{2-q}}.
\end{equation}
In order to obtain a lower bound for $A_1$, we combine~\eqref{A_n_strong},~\eqref{e_strong_opt}, and~\eqref{L_strong_opt} as follows:
\begin{equation*}
A_1 = \frac{1}{L_1} \geq \frac{(\theta_0 \epsilon_0)^{\frac{2-q}{q}}}{2L^{\frac{2}{q}}} = \frac{1}{2L^{\frac{2}{q}} A_1^{\frac{2q}{3q-2}}}.
\end{equation*}
That is, we have
$
A_1 \geq \left( \frac{1}{2L^{\frac{2}{q}}}\right)^{\frac{3q-2}{5q-2}}.
$
By~\eqref{A_n_bound_strong_opt2}, we obtain~\eqref{A_n_bound_strong_opt} by setting
\begin{equation*}
\widetilde{C} = \min \left\{ \left( \frac{1}{2L^{\frac{2}{q}}}\right)^{\frac{3q-2}{5q-2}}, \left( \frac{C^{\frac{2-q}{2q}}}{2^{\frac{3}{2}}(2^{\frac{3q-2}{2-q}} - 1) \kappa^{\frac{1}{2}}}\right)^{\frac{3q-2}{2-q}} \right\}.
\end{equation*}
The upper bound for $F(x_n) - F(x^*)$ can be obtained from Theorem~\ref{Thm:strong} and~\eqref{A_n_bound_strong_opt} as follows:
\begin{equation*} \begin{split}
F(x_n) - F(x^*) 
&\leq \frac{1}{2A_n} \left( \| x_0 - x^* \|^2 + \sum_{j=1}^n \frac{C}{A_j^{\frac{2-q}{3q-2}}} \right) \\
&\leq \frac{1}{2\widetilde{C}n^{\frac{3q-2}{2-q}}} \left( \| x_0 - x^* \|^2 + \sum_{j=1}^n \frac{C}{\widetilde{C}^{\frac{2-q}{3q-2}}j} \right) \\
&\leq \frac{1}{2 \widetilde{C}n^{\frac{3q-2}{2-q}}} \left( \|x_0 - x^* \|^2 + \frac{C}{\widetilde{C}^{\frac{2-q}{3q-2}}} \left( 1 + \log n \right) \right),
\end{split}\end{equation*}
where we used an elementary inequality
\begin{equation*}
\sum_{j=1}^n \frac{1}{j} \leq 1 + \log n, \quad n \geq 1
\end{equation*}
in the last step.
\end{proof}

We conclude this section by introducing an alternative choice of the tolerance sequence $\{ \epsilon_n \}_{n \geq 0}$ for Algorithm~\ref{Alg:strong}.
Motivated by the adaptive restart schemes proposed in~\cite{OC:2015}, we choose
\begin{equation}
\label{e_strong_ada}
\displaystyle \epsilon_n = \begin{cases} \epsilon_{n-1}, &  F(\tilde{x}_{n}) \leq F(x_{n-1}), \\
\displaystyle \frac{\epsilon_{n-1}}{2}, & F(\tilde{x}_{n}) > F(x_{n-1}),
\end{cases} \quad n \geq 1
\end{equation}
for some $\epsilon_0 > 0$.
That is, $\epsilon_n$ is halved whenever the gradient descent step is unsatisfactory in the sense that the energy does not decrease.
As stated in~\cite{OC:2015}, such a reduction scheme makes the algorithm to avoid wasted iterations that move away from the optimum.
We will present numerical results of Algorithm~\ref{Alg:strong} with~\eqref{e_strong_opt} and~\eqref{e_strong_ada} in Section~\ref{Sec:Numerical}; they show faster convergence compared to existing ones, and their convergence rates are robust with respect to variation of input parameters.

\section{Extension to uniformly convex problems}
\label{Sec:Uniform}
The goal of this section is to generalize the fast gradient methods introduced in Section~\ref{Sec:Strong} to uniformly convex problems satisfying Assumptions~\ref{Ass:smooth} and~\ref{Ass:convex}.
The main ingredient is Lemma~\ref{Lem:inexact_convex}, an analogy of Lemma~\ref{Lem:inexact_smooth} for uniformly convex functions; see also~\cite[Theorem~3]{DGN:2013}.
We note that a remark similar to Remark~\ref{Rem:inexact_smooth} can be applied to Lemma~\ref{Lem:inexact_convex}.

\begin{lemma}
\label{Lem:inexact_convex}
Under Assumption~\ref{Ass:convex}, for any $\delta > 0$, we have
\begin{equation*}
f(x) \geq f(y) + \left< \nabla f(y) , x-y \right> + \frac{\delta^{\frac{p-2}{p}} \mu^{\frac{2}{p}}}{2} \| x - y \|^2 - \frac{\delta}{2}, \quad x,y \in K_0.
\end{equation*}
\end{lemma}
\begin{proof}
The statement of the lemma becomes trivial when $p=2$; we deal with the case $p > 2$ only.
Note that $\frac{p}{2}$ and $\frac{p}{p-2}$ are H\"{o}lder conjugates.
It follows by Young's inequality that
\begin{multline*} 
\frac{\mu}{p} \| x - y \|^p + \frac{\delta}{2}
= \frac{2}{p} \left( \left( \frac{\mu}{2} \right)^{\frac{2}{p}} \| x - y \|^2 \right)^{\frac{p}{2}} + \frac{p-2}{p} \left( \left( \frac{p }{2(p-2)} \delta \right)^{\frac{p-2}{p}} \right)^{\frac{p}{p-2}} \\
\geq \left( \frac{\mu}{2} \right)^{\frac{2}{p}} \| x - y \|^2 \cdot \left(\frac{p }{2(p-2)} \delta \right)^{\frac{p-2}{p}} 
\geq \frac{\delta^{\frac{p-2}{p}} \mu^{\frac{2}{p}}}{2} \| x - y \|^2.
\end{multline*}
Combining the above inequality with Assumption~\ref{Ass:convex} completes the proof.
\end{proof}

Invoking Lemma~\ref{Lem:inexact_convex}, one can modify the momentum equation~\eqref{A_n_strong} of Algorithm~\ref{Alg:strong} to be suitable for uniformly convex problems.
The resulting algorithm is summarized in Algorithm~\ref{Alg:uniform}; it presents a novel fast gradient method for~\eqref{model} under Assumptions~\ref{Ass:smooth} and~\ref{Ass:convex}.

\begin{algorithm}[]
\caption{Fast gradient method for~\eqref{model} under Assumptions~\ref{Ass:smooth} and~\ref{Ass:convex} with variable tolerance}
\begin{algorithmic}[]
\State Choose $x_0 \in X$ and $L_0 > 0$.
\State Let $A_0 = 0$ and $\phi_0 (x) = \frac{1}{2} \| x - x_0 \|^2$.

\For{$n=0,1,2,\dots$}
\State \begin{equation*}
\hat{L}_{n+1} \leftarrow L_n / 2, \quad
v_n = \argmin_{x \in X} \phi_n (x)
\end{equation*}

\Repeat
\State Find $a_{n+1} \geq 0$ such that
\begin{equation}
\label{A_n_uniform}
\frac{a_{n+1}^2}{A_n + a_{n+1}} = \frac{1 + \sum_{j=1}^n \delta_{j-1}^{\frac{p-2}{p}} \mu^{\frac{2}{p}} a_j}{\hat{L}_{n+1}}.
\end{equation}
\State Choose $\epsilon_n > 0$, $\delta_n > 0$ and compute $\theta_n$, $y_n$, $z_n$, and $\tx_{n+1}$ by~\eqref{common_subroutine}.

\If {$\displaystyle F(\tx_{n+1}) > \ell_F(\tx_{n+1}; y_n) + \frac{\hat{L}_{n+1}}{2} \| \tx_{n+1} - y_n \|^2 + \frac{\theta_n \epsilon_n}{2}$}
\State \begin{equation*}
\hat{L}_{n+1} \leftarrow 2\hat{L}_{n+1}
\end{equation*}
\EndIf
\Until {$\displaystyle F(\tx_{n+1}) \leq \ell_F(\tx_{n+1}; y_n) + \frac{\hat{L}_{n+1}}{2} \| \tx_{n+1} - y_n \|^2 + \frac{\theta_n\epsilon_n}{2}$}
\State Pick $x_{n+1}$ such that $\displaystyle F(x_{n+1}) \leq \min \left\{ F(\tx_{n+1}), F(x_n) \right\}$.
\begin{gather*} 
L_{n+1} = \hat{L}_{n+1}, \quad 
A_{n+1} = A_n + a_{n+1}, \\
\phi_{n+1}(x) = \phi_{n}(x) + a_{n+1} \left( \ell_F (x; y_n) + \frac{ \delta_{n}^{\frac{p-2}{p}} \mu^{\frac{2}{p}}}{2} \| x - y_n \|^2 - \frac{\delta_n}{2} \right)
\end{gather*} 
\EndFor
\end{algorithmic}
\label{Alg:uniform}
\end{algorithm}

In order to obtain an abstract energy error estimate of Algorithm~\ref{Alg:uniform}, we closely follow~\cite{Nesterov:2013,Nesterov:2015}.
We first prove two properties of the sequence of estimating functions $\{ \phi_n \}_{n \geq 0}$ in Lemmas~\ref{Lem:estimating1} and~\ref{Lem:estimating2}.
Then we utilize them to get the abstract convergence theorem for Algorithm~\ref{Alg:uniform} presented in Theorem~\ref{Thm:uniform}.

\begin{lemma}
\label{Lem:estimating1}
Let $\{ A_n \}$ and $\{ \phi_n \}$ be the sequences generated by Algorithm~\ref{Alg:uniform}.
For any $n \geq$ 0, we have
\begin{equation*}
\phi_n (x) \leq A_n F(x) + \frac{1}{2} \| x - x_0 \|^2 \quad \forall x \in X.
\end{equation*}
\end{lemma}
\begin{proof}
With Lemma~\ref{Lem:inexact_convex}, the proof can be done similarly to~\cite[Lemma~7]{Nesterov:2013}.
It is obvious in the case $n = 0$.
We assume that the claim holds for some $n \geq 0$.
By Lemma~\ref{Lem:inexact_convex}, we have
\begin{equation*} \begin{split}
\phi_{n+1}(x) &= \phi_n (x) + a_{n+1} \left( \ell_F(x; y_n) + \frac{ \delta_{n}^{\frac{p-2}{p}} \mu^{\frac{2}{p}}}{2} \| x - y_n \|^2 - \frac{\delta_n}{2}\right) \\
&\leq A_n F(x) + \frac{1}{2} \| x - x_0 \|^2 + a_{n+1} F(x) \\
&= A_{n+1} F(x) + \frac{1}{2} \| x - x_0 \|^2.
\end{split} \end{equation*}
The claim is also valid for $n+1$.
By mathematical induction, it is true for all $n \geq 0$.
\end{proof}

\begin{lemma}
\label{Lem:estimating2}
Let $\{ x_n \}$, $\{ v_n \}$, $\{A_n \}$, and $\{ \phi_n \}$ be the sequences generated by Algorithm~\ref{Alg:uniform}.
For any $n \geq 0$, we have
\begin{equation*}
A_n \left( F(x_n ) - \frac{\bepsilon_n + \bdelta_n}{2} \right) \leq \phi_n (v_n ),
\end{equation*}
where $\bepsilon_n = \frac{1}{A_n}\sum_{j=1}^n a_j \epsilon_{j-1}$ and $\bdelta_n = \frac{1}{A_n}\sum_{j=1}^n a_j \delta_{j-1}$.
\end{lemma}
\begin{proof}
The claim obviously holds when $n = 0$.
Assume that it is valid for some $n \geq 0$.
We first note that
\begin{equation}
\label{epsilon_recur}
A_n \bepsilon_n + a_{n+1} \epsilon_n = A_{n+1} \bepsilon_{n+1}, \quad
A_n \bdelta_n + a_{n+1} \delta_n = A_{n+1} \bdelta_{n+1}.
\end{equation}
It is easy to check that $\phi_n$ is strongly convex with parameter $1 + \sum_{j=1}^n \delta_{j-1}^{\frac{p-2}{p}} \mu^{\frac{2}{p}} a_j$.
Since $v_n$ minimizes $\phi_n$, we get
\begin{equation*} \begin{split}
\phi_n (x) &\geq \phi_n (v_n) + \frac{1 + \sum_{j=1}^n \delta_{j-1}^{\frac{p-2}{p}} \mu^{\frac{2}{p}} a_j}{2} \| x - v_n \|^2 \\
&\geq A_n \left( \ell_F (x_n; y_n) - \frac{\bepsilon_n + \bdelta_n}{2} \right) + \frac{1 + \sum_{j=1}^n \delta_{j-1}^{\frac{p-2}{p}} \mu^{\frac{2}{p}} a_j}{2} \|x - v_n \|^2 \\
&\stackrel{\eqref{A_n_uniform}}{=} A_n \left( \ell_F (x_n; y_n) - \frac{\bepsilon_n + \bdelta_n}{2} \right) + \frac{a_{n+1} \theta_n L_{n+1}}{2} \|x - v_n \|^2
\end{split} \end{equation*}
for any $x \in X$.
It follows by the definitions of $\phi_{n+1}$ and $z_n$ that
\begin{equation*}
\resizebox{\linewidth}{!}{$
\begin{split}
\phi_{n+1}&(v_{n+1}) \geq A_n \left( \ell_F (x_n; y_n) - \frac{\bepsilon_n + \bdelta_n}{2} \right) + \frac{a_{n+1} \theta_n L_{n+1}}{2} \|v_{n+1} - v_n \|^2
+ a_{n+1}  \left( \ell_F ( v_{n+1}; y_n) - \frac{\delta_n}{2} \right) \\
&\stackrel{\eqref{epsilon_recur}}{\geq} A_n \left( \ell_F (x_n; y_n) - \frac{\bepsilon_n}{2} \right) + \frac{a_{n+1} \theta_n L_{n+1}}{2} \| z_n - v_n \|^2
 + a_{n+1} \ell_F ( z_n ; y_n) - \frac{A_{n+1} \bdelta_{n+1}}{2} \\
&= \frac{a_{n+1} \theta_n L_{n+1}}{2} \| z_n - v_n \|^2 + A_{n+1}\left( (1 - \theta_n) \ell_F (x_n ; y_n) + \theta_n \ell_F (z_n ; y_n) \right)
- \frac{A_n \bepsilon_n}{2} - \frac{A_{n+1} \bdelta_{n+1}}{2}.
\end{split} $}
\end{equation*}
Noting that $\tx_{n+1} - y_n = \theta_n (z_n - v_n)$ and $\tx_{n+1} = (1-\theta_n) x_n + \theta_n z_n$, we have
\begin{equation*} \begin{split}
\phi_{n+1}(v_{n+1}) &\geq \frac{a_{n+1} L_{n+1}}{2\theta_n} \| \tx_{n+1} - y_n \|^2 + A_{n+1} \ell_F (\tx_{n+1} ; y_n) - \frac{A_n \bepsilon_n}{2} - \frac{A_{n+1} \bdelta_{n+1}}{2}\\
&= A_{n+1} \left( \ell_F (\tx_{n+1};y_n) + \frac{L_{n+1}}{2} \| \tx_{n+1} - y_n \|^2 - \frac{\bdelta_{n+1}}{2} \right) - \frac{A_n \bepsilon_n}{2} \\
&\geq A_{n+1} \left( F(\tx_{n+1}) - \frac{\theta_n \epsilon_{n} + \bdelta_{n+1}}{2} \right) - \frac{A_n \bepsilon_n}{2} \\
&\stackrel{\eqref{epsilon_recur}}{\geq} A_{n+1} \left( F(x_{n+1}) - \frac{\bepsilon_{n+1} + \bdelta_{n+1}}{2} \right),
\end{split} \end{equation*}
where the penultimate inequality is due to the backtracking condition.
Hence, the claim is also valid for $n+1$ and we complete the proof by mathematical induction.
\end{proof}

\begin{theorem}
\label{Thm:uniform}
In Algorithm~\ref{Alg:uniform}, we have
\begin{equation*}
F(x_n ) - F(x^* ) \leq  \frac{\| x_0 - x^* \|^2}{2A_n} + \frac{\bepsilon_n + \bdelta_n}{2}, \quad n \geq 1,
\end{equation*}
where $\bepsilon_n = \frac{1}{A_n}\sum_{j=1}^n a_j \epsilon_{j-1}$ and $\bdelta_n = \frac{1}{A_n}\sum_{j=1}^n a_j \delta_{j-1}$.
\end{theorem}
\begin{proof}
Invoking Lemma~\ref{Lem:estimating1} with $x = x^*$ yields
\begin{equation*}
\phi (x^*) \leq A_n F(x^* ) + \frac{1}{2} \| x_0 - x^* \|^2.
\end{equation*}
By combining the above equation with Lemma~\ref{Lem:estimating2}, we get
\begin{equation*}
F(x_n ) - F(x^* ) \leq  \frac{\| x^* - x_0 \|^2}{2A_n} + \frac{\bepsilon_n + \bdelta_n}{2} ,
\end{equation*}
which completes the proof.
\end{proof}

\begin{remark}
\label{Rem:zero}
If $q = 2$, then Lemma~\ref{Lem:inexact_smooth} holds for $\epsilon = 0$ with the convention $0^0 = 1$.
In this case, one can easily verify that setting $\epsilon_n = 0$ for all $n \geq 0$ in Algorithm~\ref{Alg:uniform} does not alter the result of Theorem~\ref{Thm:uniform}.
Similarly, Lemma~\ref{Lem:inexact_convex} holds for $\delta = 0$ when $p= 2$ and one may set $\delta_n = 0$ for all $n \geq 0$ in Algorithm~\ref{Alg:uniform}.
Indeed, Algorithm~\ref{Alg:uniform} reduces to Algorithm~\ref{Alg:strong} if $p = 2$ and $\delta_n$ is chosen as 0.
\end{remark}

We present some notable choices of the tolerance sequences $\{ \epsilon_n \}_{n \geq 0}$ and $\{ \delta_n \}_{n \geq 0}$.
Similar to Algorithm~\ref{Alg:strong0}, we first consider the case of constant tolerances, i.e., we set
\begin{equation}
\label{e_uniform_const}
\epsilon_n = \delta_n = \frac{\epsilon}{2} > 0, \quad n \geq 0
\end{equation}
for some $\epsilon > 0$.
In the following, we show that Algorithm~\ref{Alg:uniform}-\eqref{e_uniform_const} requires at most $O(\epsilon^{-\frac{2(p-q)}{p(3q-2)}} \log \epsilon^{-1})$ iterations to get an $\epsilon$-solution.
This bound is optimal up to a logarithmic factor under Assumptions~\ref{Ass:smooth} and~\ref{Ass:convex} according to~\cite{NN:1985}.
The overall derivation procedure is similar to the case of Algorithm~\ref{Alg:strong0}.

\begin{lemma}
\label{Lem:backtracking_uniform_const}
In Algorithm~\ref{Alg:uniform}-\eqref{e_uniform_const}, the backtracking process terminates in finite steps.
Moreover, we have
\begin{equation}
\label{L_uniform_const}
L_{n+1} \leq \frac{(2L)^{\frac{2}{q}}}{(\theta_n \epsilon )^{\frac{2-q}{q}}}, \quad n \geq 0.
\end{equation}
\end{lemma}
\begin{proof}
The proof can be done by the same argument as~\cite[Proposition~A.2]{RD:2020}.
\end{proof}

\begin{theorem}
\label{Thm:uniform_const}
In Algorithm~\ref{Alg:uniform}-\eqref{e_uniform_const}, we have
\begin{equation}
\label{A_n_bound_uniform_const}
A_n \geq \frac{\epsilon^{\frac{2-q}{q}}}{(2L)^{\frac{2}{q}}} \left( 1 + \frac{\epsilon^{\frac{2(p-q)}{p(3q-2)}}}{2^{\frac{2q(2p-1)}{p(3q-2)}} \kappa^{\frac{q}{3q-2}}} \right)^{\frac{3q-2}{q}(n-1)},
\quad n \geq 1,
\end{equation}
where $\kappa$ was given in~\eqref{kappa}.
Consequently, we need at most
\begin{equation*}
 1 + \frac{q}{3q-2} \frac{
\log \left( (2L)^{\frac{2}{q}}  \| x_0 - x^* \|^2 \right) + \frac{2}{q} \log \frac{1}{\epsilon}}{\log \left( 1 + \frac{\epsilon^{\frac{2(p-q)}{p(3q-2)}}}{2^{\frac{2q(2p-1)}{p(3q-2)}} \kappa^{\frac{q}{3q-2}}} \right)} = O \left( \frac{\kappa^{\frac{q}{3q-2}} }{\epsilon^{\frac{2(p-q)}{p(3q-2)}}}  \log \frac{1}{\epsilon} \right)
\end{equation*}
iterations to obtain an $\epsilon$-solution of~\eqref{model}.
\end{theorem}
\begin{proof}
This proof closely follows the one of Theorem~\ref{Thm:strong0}.
Under the setting $\epsilon_n = \delta_n = \frac{\epsilon}{2}$, the recurrence inequality~\eqref{A_n_uniform} reduces to
\begin{equation}
\label{A_n_uniform_const}
\frac{(A_{n+1} - A_n)^2}{A_{n+1}} = \frac{1 + \left(\frac{\epsilon}{2} \right)^{\frac{p-2}{p}} \mu^{\frac{2}{p}} A_n}{L_{n+1}}.
\end{equation}
Applying Lemma~\ref{Lem:backtracking_uniform_const} to~\eqref{A_n_uniform_const} that
\begin{equation}
\label{recur_uniform_const}
(A_{n+1} - A_n)^{\frac{3q-2}{q}} \geq \frac{\epsilon^{\frac{2(p-q)}{pq}}}{2^{1 + \frac{2(p-q)}{pq}}\kappa}  A_n A_{n+1}^{\frac{2q-2}{q}}.
\end{equation}
On the other hand, substituting $n=0$ into~\eqref{A_n_uniform_const} yields
$
A_1 = \frac{1}{L_1 } \geq \frac{\epsilon^{\frac{2-q}{q}}}{(2L)^{\frac{2}{q}}}.
$
Combining~\eqref{recur_uniform_const} with Lemma~\ref{Lem:recur_linear}, we get~\eqref{A_n_bound_uniform_const}.

Next, we compute the number of iterations required to obtain an $\epsilon$-solution; Theorem~\ref{Thm:uniform} says that it is enough to solve the following equation with respect to $n$:
\begin{equation*}
\frac{\| x_0 - x^* \|^2}{2} = \frac{\epsilon}{2} \frac{\epsilon^{\frac{2-q}{q}}}{(2L)^{\frac{2}{q}}} \left( 1 + \frac{\epsilon^{\frac{2(p-q)}{p(3q-2)}}}{2^{\frac{2q(2p-1)}{p(3q-2)}} \kappa^{\frac{q}{3q-2}}} \right)^{\frac{3q-2}{q}(n-1)}.
\end{equation*}
The solution of the above equation is given by
\begin{equation*} \begin{split}
n &= 1 + \frac{q}{3q-2} \frac{
\log \left( (2L)^{\frac{2}{q}}  \| x_0 - x^* \|^2 \right) + \frac{2}{q} \log \frac{1}{\epsilon}}{\log \left( 1 + \frac{\epsilon^{\frac{2(p-q)}{p(3q-2)}}}{2^{\frac{2q(2p-1)}{p(3q-2)}} \kappa^{\frac{q}{3q-2}}} \right)} \\
&\leq 1 + \frac{q}{3q-2} \left( \log \left( (2L)^{\frac{2}{q}}  \| x_0 - x^* \|^2 \right) + \frac{2}{q} \log \frac{1}{\epsilon} \right) \left( \frac{1}{2} + \frac{2^{\frac{2q(2p-1)}{p(3q-2)}} \kappa^{\frac{q}{3q-2}}}{\epsilon^{\frac{2(p-q)}{p(3q-2)}}} \right), 
\end{split} \end{equation*}
where we used~\eqref{log_ineq} in the inequality.
\end{proof}

Next, generalizing~\eqref{e_strong_opt}, we choose
\begin{equation}
\label{e_uniform_opt}
\epsilon_n = \frac{C_{\epsilon}}{a_{n+1} (A_n + a_{n+1})^{\frac{2(p-q)}{p(3q-2)}}}, \quad \delta_n = \frac{C_{\delta}}{a_{n+1} (A_n + a_{n+1})^{\frac{2(p-q)}{p(3q-2)}}},
\quad n \geq 0
\end{equation}
for some $C_\epsilon, C_{\delta} > 0$.
Due to very complicated structure of the recurrence relation for $A_n$~(see~\eqref{recur_complex}), we were not able to complete a rigorous convergence analysis in this case.
However, we found a sufficient condition~(see Proposition~\ref{Prop:other}) for almost optimal $O(n^{-\frac{p(3q-2)}{2(p-q)}} \log n)$ convergence of Algorithm~\ref{Alg:uniform}-\eqref{e_uniform_opt}, and verified the condition numerically; relevant results are given in Appendix~\ref{App:Claim}.
The almost optimal convergence of Algorithm~\ref{Alg:uniform}-\eqref{e_uniform_opt} is summarized in Claim~\ref{Claim:uniform_opt}.
We also mention that Claim~\ref{Claim:uniform_opt} can be proven rigorously for some special cases, e.g., when $p=2$.

\begin{claim}
\label{Claim:uniform_opt}
In Algorithm~\ref{Alg:uniform}-\eqref{e_uniform_opt}, it satisfies that
\begin{equation*}
F(x_n) - F(x^*) \leq \begin{cases} \displaystyle O \left( \frac{\log n}{n^{\frac{p(3q-2)}{2(p-q)}}} \right), & p > q, \\
\displaystyle O \left( n \left( 1 + \frac{1}{2^{\frac{3}{2}} \kappa^{\frac{1}{2}}} \right)^{-2n}\right) , & p=q, \end{cases}
\quad n \geq 1.
\end{equation*}
\end{claim}

Finally, we present an adaptive choice of the tolerance sequences with respect to energy decrease:
\begin{equation}
\label{e_uniform_ada}
\displaystyle \epsilon_n = \begin{cases} \epsilon_{n-1}, &  F(\tilde{x}_{n}) \leq F(x_{n-1}), \\
\displaystyle \frac{\epsilon_{n-1}}{2}, & F(\tilde{x}_{n}) > F(x_{n-1}),
\end{cases}
\quad
\displaystyle \delta_n = \begin{cases} \delta_{n-1}, &  F(\tilde{x}_{n}) \leq F(x_{n-1}), \\
\displaystyle \frac{\delta_{n-1}}{2}, & F(\tilde{x}_{n}) > F(x_{n-1}),
\end{cases}
 \quad n \geq 1
\end{equation}
for some $\epsilon_0, \delta_0 > 0$.
It is clear that~\eqref{e_uniform_ada} generalizes~\eqref{e_strong_ada}.
Numerical results of Algorithm~\ref{Alg:uniform} with various choices of the tolerance sequences as~\eqref{e_uniform_const},~\eqref{e_uniform_opt}, and~\eqref{e_uniform_ada}, and relevant discussions will be provided in Section~\ref{Sec:Numerical}.

\section{Numerical results}
\label{Sec:Numerical}
In this section, we conduct numerical experiments of the fast gradient methods introduced in this paper applied to a finite element approximation of the $s$-Laplacian problem~\eqref{sLap}~\cite{Ciarlet:2002}.
All the algorithms were implemented in MATLAB R2020b, and executed on a desktop equipped with Intel Core i5-8600K CPU~(3.60GHz), 32GB RAM, and Windows~10 Pro 64-bit OS.

In~\eqref{sLap}, we set $\Omega = [0,1]^2 \subset \mathbb{R}^2$.
The domain $\Omega$ is partitioned into $2 \times 1/h \times 1/h$ uniform triangles to form a triangulation $\mathcal{T}_h$ of $\Omega$.
The piecewise linear and continuous finite element space on $\mathcal{T}_h$ is denoted by $S_h$.
A conforming finite element discretization of~\eqref{sLap} using $S_h \subset W_0^{1,s} (\Omega)$ is written as
\begin{equation}
\label{sLap_FEM}
\min_{u \in S_h } \left\{F(u) := \frac{1}{s} \int_{\Omega} \lvert \nabla u \rvert^s \,dx - \int_{\Omega} bu \,dx \right\}.
\end{equation}
The problem~\eqref{sLap_FEM} is represented in terms of~\eqref{model} with
\begin{equation*}
X = S_h, \quad f(u) = \frac{1}{s} \int_{\Omega} \lvert \nabla u \rvert^s \,dx - \int_{\Omega} bu \,dx, \quad g(u) = 0.
\end{equation*}
It is well-known that~(see, e.g.,~\cite{Ciarlet:2002,Park:2020}) the energy function $F$ satisfies Assumptions~\ref{Ass:smooth} and~\ref{Ass:convex} with $q = \min \left\{2, s \right\}$ and $p = \max \left\{2, s \right\}$, respectively.

\begin{figure}[]
\centering
\subfloat[][Tolerance~\eqref{e_strong_opt}, varying $C$]{ \includegraphics[width=0.45\linewidth]{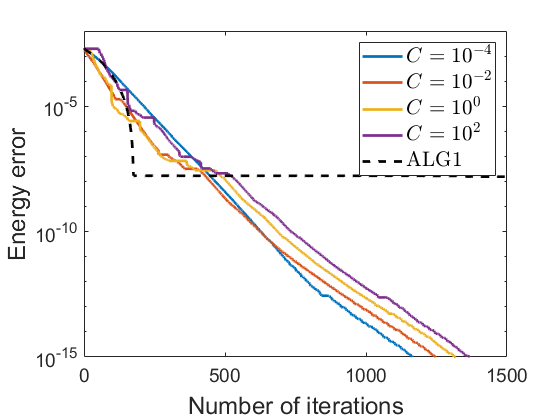} 
\label{Fig:strong(a)}}
\subfloat[][Tolerance~\eqref{e_strong_ada}, varying $\epsilon_0$]{ \includegraphics[width=0.45\linewidth]{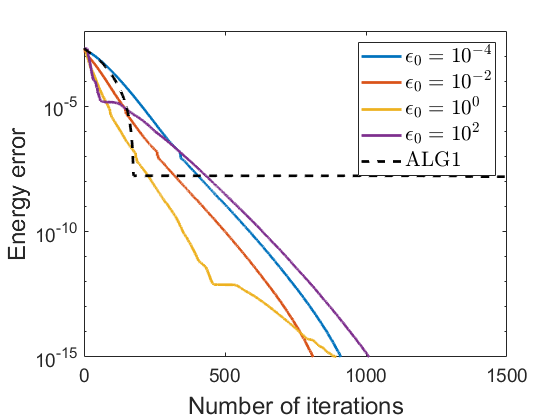} 
\label{Fig:strong(b)}}
\caption{Decay of the energy error $F(x_n) - F(x^*)$ of Algorithm~\ref{Alg:strong}~($x_0 = 0$, $L_0 = 1$) for~\eqref{sLap_FEM}~($s=1.5$, $b=1$, $h=2^{-5}$) equipped with the tolerance sequences \textbf{(a)}~\eqref{e_strong_opt} with respect to various $C$ and \textbf{(b)}~\eqref{e_strong_ada} with respect to various $\epsilon_0$.
The black dashed line refers to Algorithm~\ref{Alg:original}~($x_0 = 0$, $L_0 = 1$, $\epsilon = 10^{-10}$).}
\label{Fig:strong}
\end{figure}

\begin{figure}[]
\centering
\includegraphics[width=0.45\linewidth]{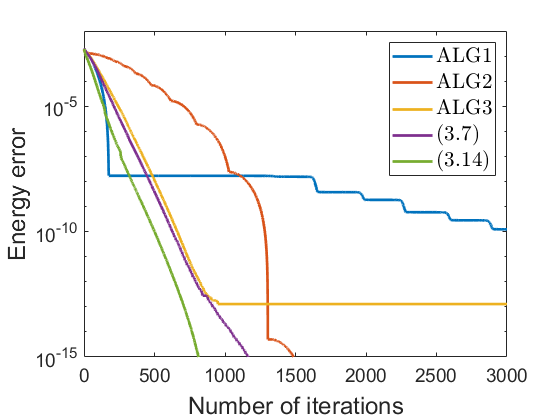} 
\caption{Comparison of various fast gradient methods for~\eqref{sLap_FEM}~($s=1.5$, $b=1$, $h=2^{-5}$).
For all algorithms we set $x_0 = 0$ and $L_0 = 1$.
Five curves refer to Algorithm~\ref{Alg:original}~($\epsilon = 10^{-10}$), Algorithm~\ref{Alg:scheduled}~($\epsilon_0 = e^{-\frac{3q-2}{2}} \left( F(x_0) - F(x^*) \right)$, $C = 2$), Algorithm~\ref{Alg:strong0}~($\epsilon = 10^{-10}$), Algorithm~\ref{Alg:strong}-\eqref{e_strong_opt}~($C=10^{-4}$), and Algorithm~\ref{Alg:strong}-\eqref{e_strong_ada}~($\epsilon_0 = 10^{-2}$), respectively.}
\label{Fig:comp_strong}
\end{figure}

In all numerical experiments of the algorithms for strongly convex problems, we set $s = 1.5$, $b = 1$, and $h = 2^{-5}$ in~\eqref{sLap_FEM}, so that we have $p = 2$ and $q=1.5$.
A reference solution $u^*$ of~\eqref{sLap_FEM} was computed by $10^6$ iterations of Algorithm~\ref{Alg:original} with $\epsilon = 10^{-24}$.
Initial guesses for the solution $u_0$ and the descent step size $L_0$ were chosen as $u_0 = 0$ and $L_0 = 1$, respectively.
The strong convexity parameter $\mu$ in Assumption~\ref{Ass:strong} was chosen heuristically as $0.046$.
The same setting was used for the numerical experiments corresponding to Figures~\ref{Fig:scheduled} and~\ref{Fig:strong0}.

Figure~\ref{Fig:strong} shows the performance of Algorithm~\ref{Alg:strong} for~\eqref{sLap_FEM} with $s=1.5$ under various conditions.
Decay of the energy error $F(x_n) - F(x^*)$ of Algorithm~\ref{Alg:strong} with the tolerance sequence~\eqref{e_strong_opt} with respect to $C=10^{-4}$, $10^{-2}$, $10^0$, and $10^2$.
For all cases, one readily observes that Algorithm~\ref{Alg:strong}-\eqref{e_strong_opt} outperforms Algorithm~\ref{Alg:original}.
Moreover, the convergence rate of Algorithm~\ref{Alg:strong}-\eqref{e_strong_opt} seems almost constant for all values of $C$.
It means that Algorithm~\ref{Alg:strong}-\eqref{e_strong_opt} is robust with respect to a choice of $C$.
We note that, although the smaller $C$ results the faster convergence rate, it additionally requires $O(\log C^{-1})$ calls of oracle in the backtracking process in general~(cf.~\cite{Nesterov:2015}).
On the other hand, Figure~\ref{Fig:strong(b)} shows the convergence behavior of Algorithm~\ref{Alg:strong} with the tolerance sequence~\eqref{e_strong_ada} for $\epsilon_0 = 10^{-4}$, $10^{-2}$, $10^0$, and $10^2$.
Similar to the case of~\eqref{e_strong_opt}, Algorithm~\ref{Alg:strong}-\eqref{e_strong_ada} shows the improved and robust convergence behavior with respect to variation of $\epsilon_0$.

Convergence curves of various fast gradient methods for~\eqref{sLap_FEM} with $s=1.5$ are depicted in Figure~\ref{Fig:comp_strong}.
The following five algorithms are considered: Algorithm~\ref{Alg:original}~($\epsilon = 10^{-10}$), Algorithm~\ref{Alg:scheduled}~($\epsilon_0 = e^{-\frac{3q-2}{2}} \left( F(x_0) - F(x^*) \right)$, $C = 2$), Algorithm~\ref{Alg:strong0}~($\epsilon = 10^{-10}$), Algorithm~\ref{Alg:strong}-\eqref{e_strong_opt}~($C=10^{-4}$), and Algorithm~\ref{Alg:strong}-\eqref{e_strong_ada}~($\epsilon_0 = 10^{-2}$), while we used $x_0 = 0$ and $L_0 = 1$ for all algorithms.
We can observe that Algorithm~\ref{Alg:strong} with either~\eqref{e_strong_opt} or~\eqref{e_strong_ada} outperforms the other ones in the sense of energy convergence.
Algorithm~\ref{Alg:strong0} converges as fast as Algorithm~\ref{Alg:strong}-\eqref{e_strong_opt} until first 1000 iterations but then its energy error does not decrease anymore, reflecting Theorem~\ref{Thm:strong0}.
It is remarkable that Algorithm~\ref{Alg:strong}-\eqref{e_strong_opt} converges faster than Algorithm~\ref{Alg:scheduled} although the theoretical convergence rate in Theorem~\ref{Thm:strong_opt} lacks a logarithmic factor compared to Proposition~\ref{Prop:scheduled}.
Between two tolerance sequences~\eqref{e_strong_opt} and~\eqref{e_strong_ada},~\eqref{e_strong_ada} shows faster convergence than the other.
On the other hand,~\eqref{e_strong_opt} has an advantage that its convergence is theoretically guaranteed by Theorem~\ref{Thm:strong_opt}.

\begin{figure}[]
\centering
\subfloat[][Tolerance~\eqref{e_uniform_opt}, $C_{\epsilon} = 0$, varying $C_{\delta}$]{ \includegraphics[width=0.45\linewidth]{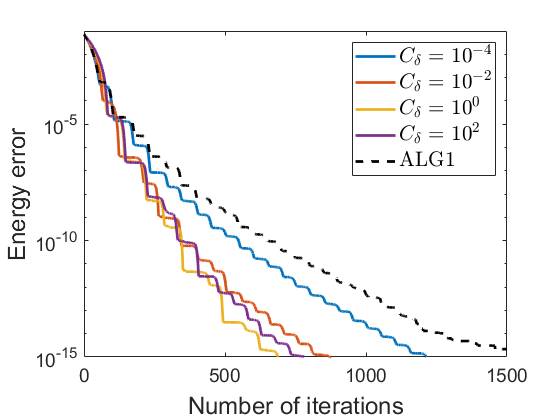} 
\label{Fig:uniform(a)}}
\subfloat[][Tolerance~\eqref{e_uniform_ada}, $\epsilon_0 = 0$, varying $\delta_0$]{ \includegraphics[width=0.45\linewidth]{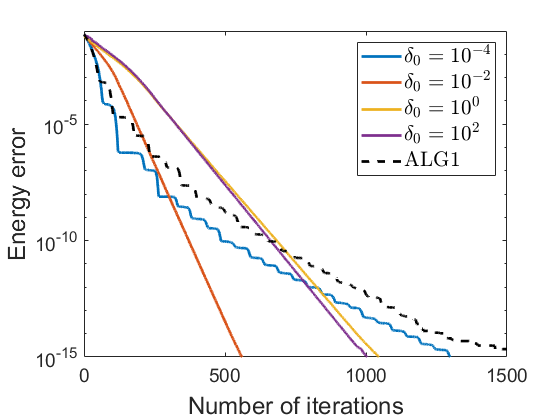} 
\label{Fig:uniform(b)}}
\caption{Decay of the energy error $F(x_n) - F(x^*)$ of Algorithm~\ref{Alg:uniform}~($x_0 = 0$, $L_0 =1$) for~\eqref{sLap_FEM}~($s=4$, $b=1$, $h=2^{-5}$) equipped with the tolerance sequences \textbf{(a)}~\eqref{e_uniform_opt} with respect to various $C_{\delta}$ and \textbf{(b)}~\eqref{e_uniform_ada} with respect to various $\delta_0$.
The black dashed line refers to Algorithm~\ref{Alg:original}~($x_0 = 0$, $L_0 = 1$, $\epsilon = 10^{-10}$).}
\label{Fig:uniform}
\end{figure}

\begin{figure}[]
\centering
\includegraphics[width=0.45\linewidth]{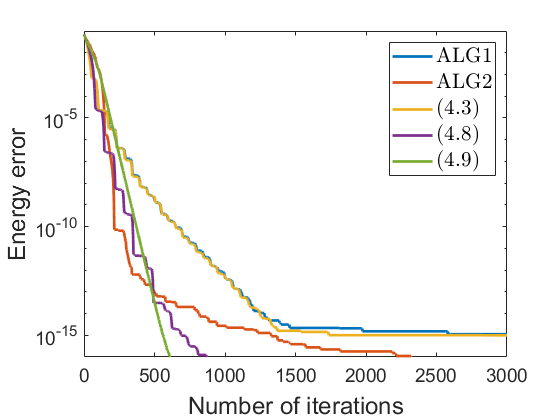} 
\caption{Comparison of various fast gradient methods for~\eqref{sLap_FEM}~($s=4$, $b=1$, $h=2^{-5}$).
For all algorithms we set $x_0 = 0$ and $L_0 = 1$.
Five curves refer to Algorithm~\ref{Alg:original}~($\epsilon = 10^{-10}$), Algorithm~\ref{Alg:scheduled}~($\epsilon_0 = e^{-\frac{3q-2}{2}} \left( F(x_0) - F(x^*) \right)$, $C = 2$), Algorithm~\ref{Alg:uniform}-\eqref{e_uniform_const}~($\epsilon = 10^{-10}$), Algorithm~\ref{Alg:uniform}-\eqref{e_uniform_opt}~($C_{\epsilon} = 0$, $C_{\delta}=10^{0}$), and Algorithm~\ref{Alg:uniform}-\eqref{e_uniform_ada}~($\epsilon_0 = 0$, $\delta_0 = 10^{-2}$), respectively.}
\label{Fig:comp_uniform}
\end{figure}

Next, we present numerical results of the algorithms for uniformly convex problems.
We set $s=4$, $b=1$, and $h=2^{-5}$ in~\eqref{sLap_FEM}, so that $p=1.5$ and $q=2$.
Similar to the case of strongly convex problems, we obtained a reference solution $u^*$ of~\eqref{sLap_FEM} by $10^6$ iterations of Algorithm~\ref{Alg:original} with $\epsilon = 10^{-24}$.
We also set $u_0 = 0$ and $L_0 = 1$.
The uniform convexity parameter $\mu$ in Assumption~\ref{Ass:convex} was given by $0.124$.

Figure~\ref{Fig:uniform} shows the performance of Algorithm~\ref{Alg:uniform} for~\eqref{sLap_FEM} with $s=4$ under various conditions.
As shown in Figures~\ref{Fig:uniform(a)} and~\ref{Fig:uniform(b)}, Algorithm~\ref{Alg:uniform}-\eqref{e_uniform_opt} and Algorithm~\ref{Alg:uniform}-\eqref{e_uniform_ada} outperform Algorithm~\ref{Alg:original} for all choices of $C_{\delta}$ and $\delta_0$, respectively.
The optimal value of $C_{\delta}$ for Algorithm~\ref{Alg:uniform}-\eqref{e_uniform_opt} was found as $10^0$, and the optimal value of $\delta_0$ for Algorithm~\ref{Alg:uniform}-\eqref{e_uniform_ada} was found as $10^{-2}$.
Compared to the case of strongly convex functions, the convergence rate of Algorithm~\ref{Alg:uniform} is relatively sensitive with respect to its input parameters either $C_{\delta}$ or $\delta_0$.
It is because that a choice of the tolerance sequence $\{ \delta_n \}_{n \geq 0}$ directly affects the momentum equation~\eqref{A_n_uniform}.

In Figure~\ref{Fig:comp_uniform}, we compare various fast gradient methods for~\eqref{sLap_FEM} with $s=4$: Algorithm~\ref{Alg:original}~($\epsilon = 10^{-10}$), Algorithm~\ref{Alg:scheduled}~($\epsilon_0 = e^{-\frac{3q-2}{2}} \left( F(x_0) - F(x^*) \right)$, $C = 2$), Algorithm~\ref{Alg:uniform}-\eqref{e_uniform_const}~($\epsilon = 10^{-10}$), Algorithm~\ref{Alg:uniform}-\eqref{e_uniform_opt}~($C_{\epsilon} = 0$, $C_{\delta} = 10^0$), and Algorithm~\ref{Alg:uniform}-\eqref{e_uniform_ada}~($\epsilon_0 = 0$, $\delta_0 = 10^{-2}$).
The performance of Algorithm~\ref{Alg:uniform}-\eqref{e_uniform_const} seems almost incomparable to that of Algorithm~\ref{Alg:original} despite of its almost optimal convergence property presented in Theorem~\ref{Thm:uniform_const}.
One can deduce that the convergence behavior of Algorithm~\ref{Alg:uniform}-\eqref{e_uniform_const} is barely affected by small $\delta$.
Meanwhile, Algorithm~\ref{Alg:uniform} with either~\eqref{e_uniform_opt} or~\eqref{e_uniform_ada} shows superior convergence compared to the other ones containing Algorithm~\ref{Alg:scheduled}, whose optimal convergence is guaranteed by Proposition~\ref{Prop:scheduled}.
We conclude that choosing suitable tolerance sequences in either Algorithm~\ref{Alg:strong} or Algorithm~\ref{Alg:uniform} yields very fast and efficient algorithms to solve uniformly convex and weakly smooth problems.

\section{Conclusion}
\label{Sec:Conclusion}
In this paper, we proposed novel fast gradient methods for uniformly convex and weakly smooth problems.
With suitably chosen tolerance sequences, the proposed methods showed optimal convergence rates up to logarithmic factors.
Faster convergence of the proposed methods compared to state-of-the-art methods such as the universal fast gradient method~\cite{Nesterov:2015} and its restarted variant~\cite{RD:2020} were shown by numerical experiments.

This paper leaves several important topics for future research.
A major drawback of the proposed methods is that a prior information on the uniform convexity parameters $p$ and $\mu$ in Assumption~\ref{Ass:convex} is required.
Although such a drawback also appears in almost of the existing works on momentum acceleration for strongly convex problems~\cite{CC:2019,CP:2016,Nesterov:2018}, it must be resolved in order for successful applications of the methods to various real-world problems.
For instance, one may try to design a suitable adaptive scheme to the unknown parameters $p$ and $\mu$ in a similar manner as~\cite{FQ:2019,Nesterov:2013}.
Meanwhile, we note that the choices of tolerance sequences such as~\eqref{e_strong_opt},~\eqref{e_strong_ada},~\eqref{e_uniform_opt}, and~\eqref{e_uniform_ada} are not optimized in the sense that there may exist other choices of tolerance parameters whose performances are better than them.
We expect that if we can optimize tolerance sequences in some sense, then we can obtain improved variants of proposed methods.
Such an idea of optimizing acceleration schemes was considered in several existing works~\cite{DT:2014,KF:2016,KF:2018}.

\backmatter

\bmhead{Acknowledgments}
This work was initially started with the help of Professor Donghwan Kim through meetings on acceleration schemes for first-order methods.
The author would like to thank him for his insightful comments and assistance.

\section*{Declarations}
\begin{itemize}
\item Funding: This work was supported by Basic Science Research Program through the National Research Foundation of Korea~(NRF) funded by the Ministry of Education~(2019R1A6A1A10073887).
\item Conflict of interest/Competing interests: The author declares that he has no conflicts of interest/competing interests.
\item Availability of data and materials: All data generated or analyzed during this study are included in this published article.
\item Code availability: The source code used to generate all data for the current study is available from the corresponding author upon request.
\end{itemize}

\begin{appendices}

\section{Recurrence inequalities}
\label{App:Ineq}
This appendix presents several useful recurrence inequalities used throughout this paper and their proofs.
Motivated by~\cite[Lemma~8]{Nesterov:2013} and~\cite[Theorem~3]{Nesterov:2015}, we state the following useful lemma.

\begin{lemma}
\label{Lem:recur_linear}
Let $\{ A_n \}_{n \geq 1}$ be an increasing sequence of positive real numbers that satisfies
\begin{equation*}
(A_{n+1} - A_{n})^{\gamma} \geq C A_{n}A_{n+1}^{\gamma - 1}, \quad n \geq1
\end{equation*}
for some $1\leq \gamma \leq 2$ and $C > 0$.
Then we have
\begin{equation*}
A_{n} \geq A_1 \left( 1+ \frac{C^{1/\gamma}}{2} \right)^{\gamma(n-1)}, \quad n \geq 1.
\end{equation*}
\end{lemma}
\begin{proof}
Take any $n \geq 1$.
Since $1/2 \leq 1/\gamma \leq 1$, the following inequality holds:
\begin{equation*}
A_{n+1} - A_{n} \leq ( A_{n+1}^{1- 1/\gamma}  + A_{n}^{1- 1/\gamma} ) ( A_{n+1}^{1/\gamma}  - A_{n}^{1/\gamma} ) .
\end{equation*}
It follows that
\begin{equation*} 
C A_n A_{n+1}^{\gamma - 1} 
\leq ( A_{n+1}^{1- 1/\gamma} + A_{n}^{1- 1/\gamma} )^{\gamma} ( A_{n+1}^{1/\gamma} - A_{n}^{1/\gamma} )^{\gamma}
\leq 2^{\gamma} A_{n+1}^{\gamma - 1} ( A_{n+1}^{1/\gamma} - A_{n}^{1/\gamma} )^{\gamma}.
\end{equation*}
Now, we have
\begin{equation*}
 A_{n+1}^{1/\gamma} - A_{n}^{1/\gamma} \geq \frac{C^{1/\gamma}}{2}  A_n^{1/\gamma} \quad 
 \Leftrightarrow \quad A_{n+1} \geq \left( 1 + \frac{C^{1/\gamma}}{2} \right)^{\gamma} A_n.
 \end{equation*}
We get the desired result by applying the above inequality recursively.
\end{proof}

The next lemma is useful when we prove sublinear convergence rate of some fast gradient methods.
We note that the proof of Lemma~\ref{Lem:recur_sublinear} closely follows~\cite[Lemma~1]{HLL:2007}.

\begin{lemma}
\label{Lem:recur_sublinear}
Let $\{ A_n \}_{n \geq 1}$ be an increasing sequence of positive real numbers that satisfies
\begin{equation*}
A_{n+1} - A_n \geq C A_n^{\gamma}, \quad n \geq 1
\end{equation*}
for some $0 \leq \gamma < 1$ and $C > 0$.
Then we have
\begin{equation*}
A_n \geq \min \left\{ A_1, \left( \frac{C}{2 ( 2^{\frac{1}{1-\gamma}} - 1 )} \right)^{\frac{1}{1-\gamma}} \right\} n^{\frac{1}{1-\gamma}}, \quad n \geq 1.
\end{equation*}
\end{lemma}
\begin{proof}
Take any $n \geq 1$.
Since $A_{n+1} \geq A_n$, we get
\begin{equation*}
A_{n+1} - A_n \geq CA_n^{\gamma} \geq C A_n A_{n+1}^{-(1-\gamma)},
\end{equation*}
or equivalently,
\begin{equation*}
\frac{A_{n+1}}{A_n} \geq 1 + C A_{n+1}^{-(1-\gamma)}.
\end{equation*}
Writing $A_n = B_n n^{\frac{1}{1- \gamma}}$, we have
\begin{equation}
\label{recur_sublinear1}
\frac{B_{n+1}}{B_n} \geq \left( \frac{n}{n+1} \right)^{\frac{1}{1-\gamma}} \left( 1 + \frac{C B_{n+1}^{-(1 - \gamma)}}{n+1} \right).
\end{equation}
The right-hand side of~\eqref{recur_sublinear1} is greater than or equal to~1 if and only if
\begin{equation}
\label{recur_sublinear2}
B_{n+1} \leq \left[\frac{n+1}{C} \left( \left( 1+ \frac{1}{n} \right)^{\frac{1}{1-\gamma}} - 1 \right) \right]^{-\frac{1}{1-\gamma}}
\end{equation}
From the fact that the right-hand side of~\eqref{recur_sublinear2} increases as $n$ increases, we deduce that a sufficient condition to satisfy $B_{n+1} \geq B_n$ is that
\begin{equation*}
B_{n+1} \leq \left( \frac{C}{2( 2^{\frac{1}{1-\gamma}} - 1 )} \right)^{\frac{1}{1-\gamma}}.
\end{equation*}
Then it is straightforward by mathematical induction that
\begin{equation*}
B_{n} \geq \min \left\{ B_1, \left( \frac{C}{2( 2^{\frac{1}{1-\gamma}} - 1 )} \right)^{\frac{1}{1-\gamma}} \right\}, \quad n \geq 1,
\end{equation*}
which implies the desired result.
\end{proof}

\section{Numerical verification of Claim~\ref{Claim:uniform_opt}}
\label{App:Claim}
This appendix is devoted to discussions on Claim~\ref{Claim:uniform_opt}.
We prove a special case of Claim~\ref{Claim:uniform_opt} and then present numerical evidences for the other cases.
First, we consider the situation when $p = 2$, i.e., the function $f$ in~\eqref{model} is strongly convex.

\begin{proposition}
\label{Prop:special}
Claim~\ref{Claim:uniform_opt} holds when $p=2$.
\end{proposition}
\begin{proof}
If $p = 2$, then~\eqref{A_n_uniform} and~\eqref{e_uniform_opt} reduces to~\eqref{A_n_strong} and~\eqref{e_strong_opt}, respectively.
Therefore, the desired result can be obtained by the same argument as Theorem~\ref{Thm:strong_opt}.
\end{proof}

Proposition~\ref{Prop:special} means that Claim~\ref{Claim:uniform_opt} is indeed a generalization of Theorem~\ref{Thm:strong_opt} to the case $p \geq 2$.
In the remainder of this appendix, we assume that $p >2 \geq q \geq 1$.
We show that the following claim is a sufficient condition to ensure Claim~\ref{Claim:uniform_opt}.

\begin{claim}
\label{Claim:recur}
Let $\{ A_n \}_{n \geq 1}$ be an increasing sequence of positive real numbers that satisfies
\begin{equation}
\label{recur_complex}
(A_{n+1} - A_n)^2 \geq C A_{n}^{1 - \frac{2-q}{q} \left( 1 + \frac{2(p-q)}{p(3q-2)} \right)} \sum_{j=1}^n \frac{(A_j - A_{j-1})^\frac{2}{p}}{A_j^{\frac{p-2}{p} \frac{2(p-q)}{p(3q-2)}}}, \quad n \geq 1
\end{equation}
for some $p > 2 \geq q \geq 1$ and $C > 0$.
Then we have
\begin{equation*}
A_n \geq \widetilde{C} n^{\frac{p(3q-2)}{2(p-q)}}, \quad n \geq 1,
\end{equation*}
where $\widetilde{C}$ is a positive constant depending on $p$, $q$, $C$, and $A_1$ only.
\end{claim}

\begin{remark}
\label{Rem:recur}
Replacing $A_n$, $A_{n+1} - A_n$, and $\sum_{j=1}^n \cdot$ in~\eqref{recur_complex} by $y(t)$, $y'(t)$, and $\int_0^{t} \cdot \,ds$, respectively, one can obtain the ordinary differential equation
\begin{equation}
\label{ODE}
y'(t)^2 = C y(t)^{1 - \frac{2-q}{q} \left( 1 + \frac{2(p-q)}{p(3q-2)} \right)} \int_0^t \frac{y'(s)^{\frac{2}{p}}}{y(s)^{\frac{p-2}{p} \frac{2(p-q)}{p(3q-2)}}} \,ds,
\end{equation}
which is a continuous analogue of~\eqref{recur_complex}.
If we impose the initial condition $y(0) = 0$ to~\eqref{ODE}, then we can readily verify that~\eqref{ODE} admits a solution
\begin{equation*}
y(t) = \widehat{C} t^{\frac{p(3q-2)}{2(p-q)}}, \quad t \geq 0,
\end{equation*}
where $\widehat{C}$ is an appropriate constant depending on $p$, $q$, and $C$.
That is, the solution of~\eqref{ODE} has the same growth rate as the conclusion of Claim~\ref{Claim:recur}.
\end{remark}

Even though Claim~\ref{Claim:recur} has the rather complex structure, it is in fact a generalization of~\eqref{Lem:recur_sublinear}.
If we set $p=2$ and $1 \leq q < 2$ in~\eqref{recur_complex}, then we get
\begin{equation*}
(A_{n+1} - A_n) \geq C^{\frac{1}{2}} A_n^{\frac{4q-4}{3q-2}}, \quad n \geq 1,
\end{equation*}
which has the same form as~\eqref{Lem:recur_sublinear}~(note that $0 \leq \frac{4q-4}{3q-2} < 1$ if $1\leq q < 2$).

\begin{proposition}
\label{Prop:other}
Assume that $p > 2 \geq q \geq 1$.
Claim~\ref{Claim:recur} implies Claim~\ref{Claim:uniform_opt}.
\end{proposition}
\begin{proof}
The starting point of the proof is~\eqref{A_n_uniform}; equations~\eqref{L_uniform_const} and~\eqref{e_uniform_opt} imply that
\begin{equation*} \begin{split}
(A_{n+1} - A_n)^2 &\geq \frac{A_{n+1} \sum_{j=1}^n \delta_{j-1}^{\frac{p-2}{p}} \mu^{\frac{2}{p}} (A_{j} - A_{j-1}) }{L_{n+1}} \\
&\geq \frac{C_{\delta}^{\frac{p-2}{p}}C_{\epsilon}^{\frac{2-q}{q}}}{2 \kappa} A_{n+1}^{1 - \frac{2-q}{q} \left( 1 + \frac{2(p-q)}{p(3q-2)} \right)} \sum_{j=1}^n \frac{(A_j - A_{j-1})^\frac{2}{p}}{A_j^{\frac{p-2}{p} \frac{2(p-q)}{p(3q-2)}}} \\
&\geq \frac{C_{\delta}^{\frac{p-2}{p}}C_{\epsilon}^{\frac{2-q}{q}}}{2 \kappa} A_{n}^{1 - \frac{2-q}{q} \left( 1 + \frac{2(p-q)}{p(3q-2)} \right)} \sum_{j=1}^n \frac{(A_j - A_{j-1})^\frac{2}{p}}{A_j^{\frac{p-2}{p} \frac{2(p-q)}{p(3q-2)}}} ,
\end{split} \end{equation*}
where $\kappa$ was defined in~\eqref{kappa}.
Similarly to~\eqref{Thm:strong_opt}, one can verify that $A_1$ has a lower bound depending on $p$, $q$, and $L$ only.
Hence, if we assume that Claim~\ref{Claim:recur} holds, then we can conclude that
$
A_n \geq \widetilde{C}n^{\frac{p(3q-2)}{2(p-q)}},
$
where $\widetilde{C}$ is a positive constant depending on $p$, $q$, $L$, $\mu$, $C_{\epsilon}$, and $C_{\delta}$.
It is straightforward to prove that
\begin{equation*}
F(x_n) - F(x^*) \leq \frac{1}{2 \widetilde{C}n^{\frac{p(3q-2)}{2(p-q)}}} \left( \|x_0 - x^* \|^2 + \frac{C}{\widetilde{C}^{\frac{2(p-q)}{p(3q-2)}}} \left( 1 + \log n \right) \right)
\end{equation*}
by invoking Theorem~\ref{Thm:uniform} and using the same argument as Theorem~\ref{Thm:strong_opt}.
\end{proof}

\begin{figure}[]
\centering
\subfloat[][$p=10$, $q=1$]{ \includegraphics[width=0.31\linewidth]{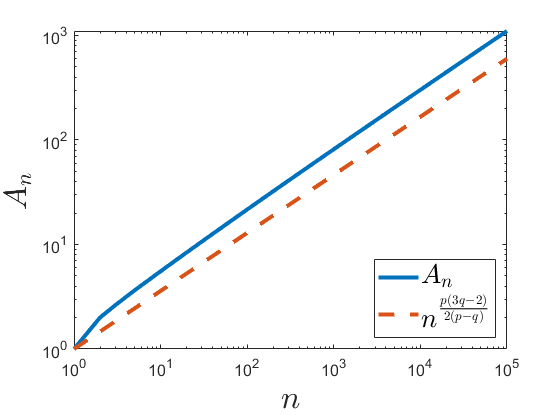} }
\subfloat[][$p=10$, $q=1.5$]{ \includegraphics[width=0.31\linewidth]{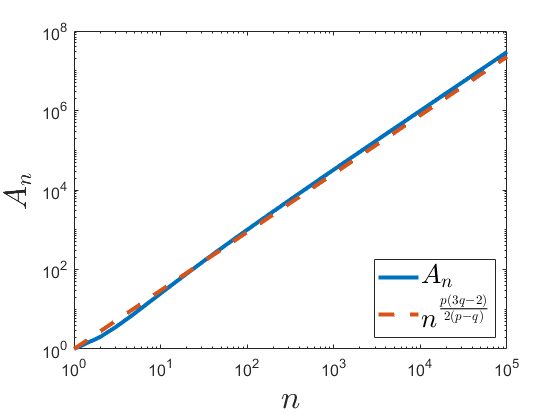} }
\subfloat[][$p=10$, $q=2$]{ \includegraphics[width=0.31\linewidth]{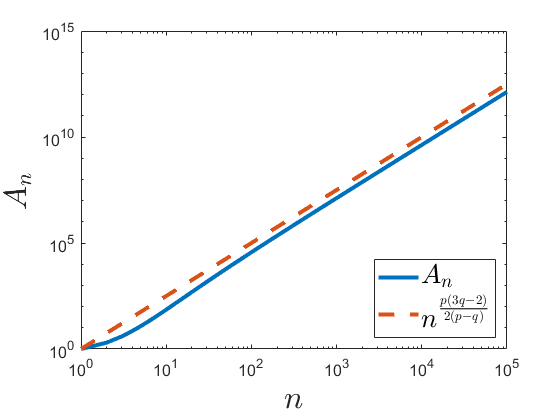} } \\

\subfloat[][$p=10^2$, $q=1$]{ \includegraphics[width=0.31\linewidth]{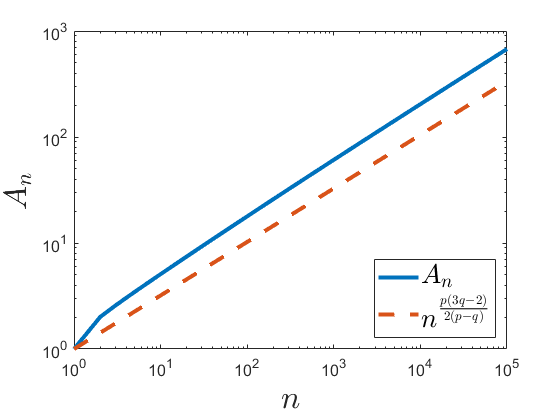} }
\subfloat[][$p=10^2$, $q=1.5$]{ \includegraphics[width=0.31\linewidth]{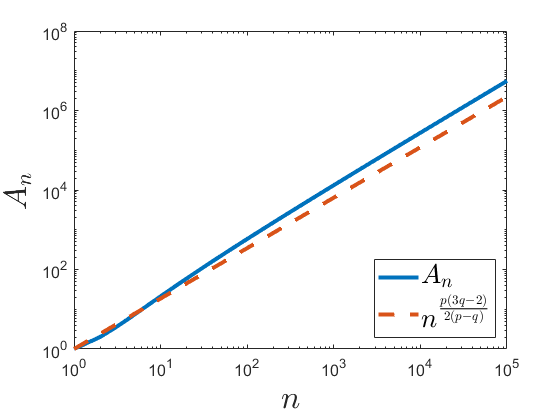} }
\subfloat[][$p=10^2$, $q=2$]{ \includegraphics[width=0.31\linewidth]{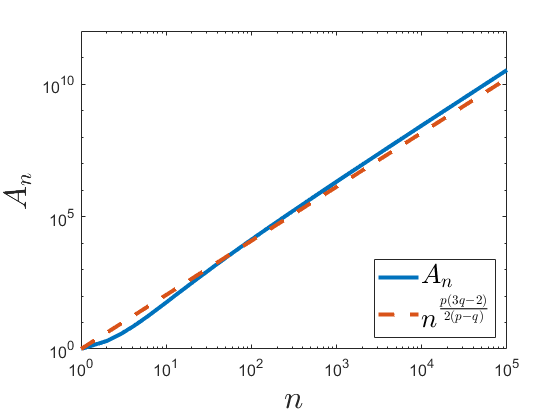} } \\

\subfloat[][$p=10^3$, $q=1$]{ \includegraphics[width=0.31\linewidth]{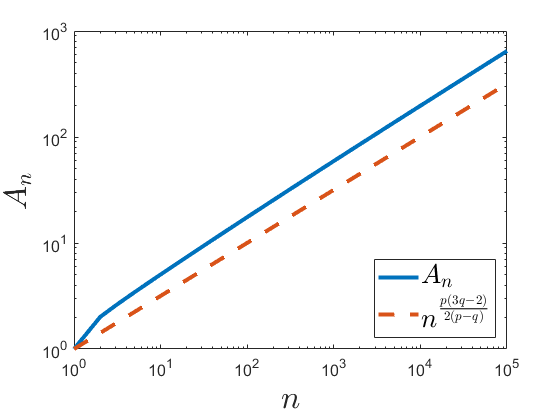} }
\subfloat[][$p=10^3$, $q=1.5$]{ \includegraphics[width=0.31\linewidth]{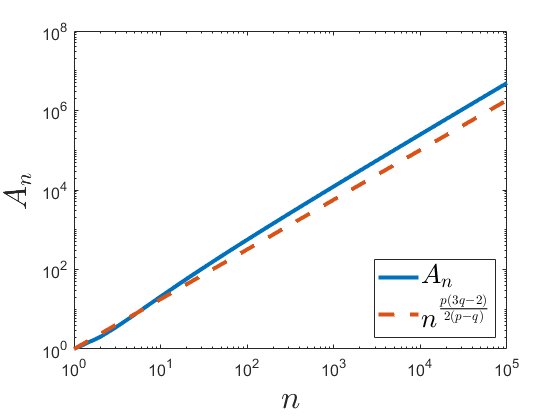} }
\subfloat[][$p=10^3$, $q=2$]{ \includegraphics[width=0.31\linewidth]{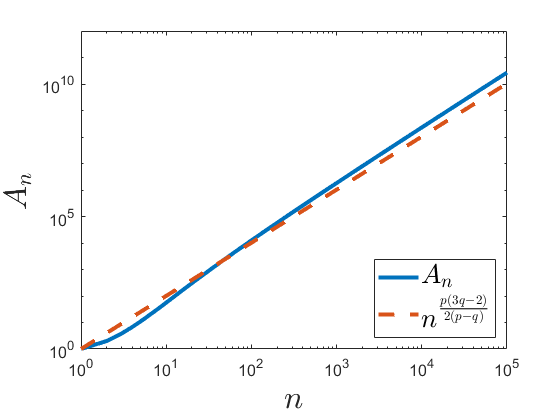} } \\
\caption{Growth of $A_n$ generated by the recurrence relation~\eqref{A_n_exp} and $n^{\frac{p(3q-2)}{2(p-q)}}$ for various $p$ and $q$ such that $p > 2 \geq q \geq 1$.}
\label{Fig:A_n}
\end{figure}

We verify Claim~\ref{Claim:recur} by numerical experiments.
Figure~\ref{Fig:A_n} plots $A_n$ and $n^{\frac{p(3q-2)}{2(p-q)}}$ with respect to $n$ in log-log scale, where $A_n$ is generated by the recurrence relation
\begin{equation} 
\label{A_n_exp}
A_1 =1, \quad
(A_{n+1} - A_n)^2 = A_{n}^{1 - \frac{2-q}{q} \left( 1 + \frac{2(p-q)}{p(3q-2)} \right)} \sum_{j=1}^n \frac{(A_j - A_{j-1})^\frac{2}{p}}{A_j^{\frac{p-2}{p} \frac{2(p-q)}{p(3q-2)}}}, \gap n \geq 1
\end{equation}
for various choices of $p$ and $q$.
In all cases, the slope of the graph for $A_n$ is a bit greater than that of the graph for $n^{\frac{p(3q-2)}{2(p-q)}}$.
That is, the asymptotic growth rate of $A_n$ is observed to be greater than $n^{\frac{p(3q-2)}{2(p-q)}}$, which implies Claim~\ref{Claim:recur}.
Thanks to Proposition~\ref{Prop:other}, the validity of Claim~\ref{Claim:uniform_opt} is ensured by the above numerical results.

\end{appendices}

\bibliography{refs_FGM_uniform}

\end{document}